\documentclass[12pt,reqno]{amsart}
\usepackage[colorlinks,urlcolor=blue,citecolor=red,pagebackref,hypertexnames=false]{hyperref}

\makeatletter
\def\@tocline#1#2#3#4#5#6#7{\relax
  \ifnum #1>\c@tocdepth 
  \else
    \par \addpenalty\@secpenalty\addvspace{#2}%
    \begingroup \hyphenpenalty\@M
    \@ifempty{#4}{%
      \@tempdima\csname r@tocindent\number#1\endcsname\relax
    }{%
      \@tempdima#4\relax
    }%
    \parindent\z@ \leftskip#3\relax \advance\leftskip\@tempdima\relax
    \rightskip\@pnumwidth plus4em \parfillskip-\@pnumwidth
    #5\leavevmode\hskip-\@tempdima
      \ifcase #1
       \or\or \hskip 1em \or \hskip 2em \else \hskip 3em \fi%
      #6\nobreak\relax
    \dotfill\hbox to\@pnumwidth{\@tocpagenum{#7}}\par
    \nobreak 
    \endgroup
  \fi}
\makeatother

\usepackage{tikz,amsthm,amsmath,amstext,amssymb,amscd,epsfig,euscript, mathrsfs, dsfont,pspicture,multicol,graphpap,graphics,graphicx,times,enumerate,subfig,sidecap,wrapfig,color}

\usepackage{ hyperref}
\usepackage{enumerate}

 \numberwithin{equation}{section}

\def\bB{{\mathbb{B}}}
\def\bC{{\mathbb{C}}}

\def\bR{{\mathbb{R}}}

\def\bN{{\mathbb{N}}}
\def\bQ{{\mathbb{Q}}}

\def\cH{{\mathscr{H}}}

\def\one{\mathds{1}}

\def\ve{\varepsilon}

\renewcommand{\d}{{\partial}}
\def\lec{\lesssim}
\def\gec{\gtrsim}

\def\ext{\mathop\mathrm{ext}} 	
 	
\DeclareMathOperator{\diam}{diam}
\def\Cap{\mathop\mathrm{Cap}} 					
\def\Tan{\mathop\mathrm{Tan}} 					
\def\dim{\mathop\mathrm{dim}} 					
\def\dist{\mathop\mathrm{dist}} 						
\def\supp{\mathop\mathrm{supp}}					

\newcommand{\ps}[1]{\left( #1 \right)}

\newcommand{\ck}[1]{\left\{#1 \right\}}
\newcommand{\av}[1]{\left| #1 \right|}

\newcommand{\cnj}[1]{\overline{#1}}

\def\XXint#1#2#3{{\setbox0=\hbox{$#1{#2#3}{\int}$ }
\vcenter{\hbox{$#2#3$ }}\kern-.58\wd0}}

\theoremstyle{plain}
\newtheorem{theorem}{Theorem}

\newtheorem{corollary}[theorem]{Corollary}

\newtheorem{lemma}[theorem]{Lemma}

\theoremstyle{definition}

\newtheorem{definition}[theorem]{Definition}

\newtheorem*{thmi}{Theorem I}
\newtheorem*{thmii}{Theorem II}

\newtheorem*{lemi}{Lemma I}

\newtheorem*{cori}{Corollary I}
\newtheorem*{corii}{Corollary II}
\newtheorem*{coriii}{Corollary III}
\newtheorem*{coriv}{Corollary IV}

\def\TheoremI{Theorem \hyperref[t:thmi]{I} }
\def\TheoremII{Theorem \hyperref[t:thmii]{II} }

\def\LemmaI{Lemma \hyperref[l:lemi]{I} }

\def\CorollaryII{Corollary \hyperref[c:corii]{II} }
\def\CorollaryIII{Corollary \hyperref[c:coriii]{III} }

\numberwithin{equation}{section}
\numberwithin{theorem}{section}

\newcommand\eqn[1]{\eqref{e:#1}}

\newcommand\Lemma[1]{Lemma \ref{l:#1}}
\newcommand\Corollary[1]{Corollary \ref{c:#1}}

  \DeclareFontFamily{U}{mathb}{\hyphenchar\font45} 
\DeclareFontShape{U}{mathb}{m}{n}{
      <5> <6> <7> <8> <9> <10> gen * mathb
      <10.95> mathb10 <12> <14.4> <17.28> <20.74> <24.88> mathb12
      }{}
\DeclareSymbolFont{mathb}{U}{mathb}{m}{n}

\DeclareMathSymbol{\toitself}      {3}{mathb}{"FD}  

\begin{document}

\title{Tangent measures and absolute continuity of harmonic measure}

\author{Jonas Azzam}
\address{Departament de Matem\`atiques\\ Universitat Aut\`onoma de Barcelona \\ Edifici C Facultat de Ci\`encies\\
08193 Bellaterra (Barcelona) }
\email{jazzam "at" mat.uab.cat}
\author{Mihalis Mourgoglou}
\address{Departament de Matem\`atiques\\ Universitat Aut\`onoma de Barcelona and Centre de Reserca Matem\` atica\\ Edifici C Facultat de Ci\`encies\\
08193 Bellaterra (Barcelona) }
\email{mmourgoglou@crm.cat}
\thanks{The two authors were supported by the ERC grant 320501 of the European Research Council (FP7/2007-2013).}

\maketitle

\begin{abstract}
We show that for uniform domains $\Omega\subseteq \bR^{d+1}$ whose boundaries satisfy a certain nondegeneracy condition that harmonic measure cannot be mutually absolutely continuous with respect to $\alpha$-dimensional Hausdorff measure unless $\alpha\leq d$. We employ a lemma that shows that, at almost every non-degenerate point, we may find a tangent measure of harmonic measure whose support is the boundary of yet another uniform domain whose harmonic measure resembles the tangent measure. 
\end{abstract}

\tableofcontents

\section{Introduction}

In this paper we discuss when the harmonic measure $\omega_{\Omega}$ for a domain $\Omega\subseteq \bR^{d+1}$ can be mutually absolutely continuous with respect to some Hausdorff measure $\cH^{\alpha}$. This is a popular problem in the case $\alpha=d$. For a simply connected planar domain $\Omega\subseteq \bC$, $ \omega_{\Omega}\ll \cH^{1}|_{\d\Omega}\ll \omega_{\Omega}$ if and only if $\d\Omega$ is a rectifiable curve by the F. and M. Riesz theorem \cite{RR16} (also see \cite{Harmonic-Measure}). In higher dimensions, some extra geometric assumptions on the domain are necessary due to counterexamples by Wu and Ziemer \cite{Wu86,Z74}. Building on work of Dahlberg \cite{Dah77}, David and Jerison showed in \cite{DJ90} that harmonic measure is in fact $A_{\infty}$-equivalent to $\cH^{d}|_{\d\Omega}$ if $\Omega\subseteq \bR^{d+1}$ is a non-tangentially accessible domain with Ahlfors $d$-regular boundary.

\begin{definition}
We say $\Omega\subseteq \bR^{d+1}$ is a {\it $C$-uniform domain} if, for every $x,y\in \cnj{\Omega}$ there is a path $\gamma\subseteq\Omega$ connecting $x$ and $y$ such that
\begin{enumerate}
\item the length of $\gamma$ is at most $C|x-y|$ and
\item for $t\in \gamma$, $\dist(t,\d\Omega)\geq \dist(t,\{x,y\})/C$. 
\end{enumerate}
A curve satisfying the above conditions is called a {\it good curve} for $x$ and $y$ in $\Omega$. We say $\Omega$ satisfies the {\it $C$-interior corkscrew condition} if for all $\xi\in \d \Omega$ and $r\in(0, \diam\d\Omega)$ there is a ball $B(x,r/C)\subseteq \Omega\cap B(\xi,r)$. 

 If $B=B(\xi,r)$, we call $B'=B(x,r/C)$ the {\it corkscrew ball} of $B$ and denote its center by $x_{B}$. We say $\Omega$ satisfies the {\it $C$-exterior corkscrew condition} if there is a ball $B(y,r/C)\subseteq B(\xi,r)\backslash\Omega$ for all $\xi\in \d\Omega$ and $r\in(0,\diam \d\Omega)$. A domain $\Omega\subseteq \bR^{d+1}$ is {\it $C$-non-tangentially accessible} (or {\it $C$-NTA}) if it has the uniform, exterior and interior corkscrew properties with constants $C$.
\label{d:uniform}
\label{d:cork}
\end{definition}

Our definition of NTA domains is slightly different than that introduced by Jerison and Kenig in \cite{JK82}, but it is equivalent \cite{AHMNT}. The appeal of these domains aside from their nice geometry are the convenient scale invariant properties of harmonic measure like being doubling. However, many of these properties have been generalized to other domains, see for example \cite{Aik01}, \cite{AH08}, and \cite{MT15}.

One still has $\cH^{d}|_{\d\Omega}\ll\omega_{\Omega}$ if $\Omega\subseteq \bR^{d+1}$ is NTA and we just assume $\cH^{d}|_{\d\Omega}$ is locally finite instead of Ahlfors $d$-regular \cite{Bad12} (or even when $\Omega$ is just uniform with rectifiable boundary \cite{Mou15}), but we don't get mutual absolute continuity \cite{AMT15}. For the most part, all these results require either assuming or establishing some rectifiability properties of the boundary of $\Omega$. Recently it was shown that rectifiability is actually necessary to have $\omega_{\Omega}\ll \cH^{d}$ even on a subset of $\d\Omega$ \cite{AHMMMTV15}. See also \cite{ABHM15} and \cite{HMIV}.

The focus for us, however, will be on the relationship between harmonic measure and $\cH^{\alpha}$ for $\alpha\neq d$, Makarov showed that for simply connected planar domains we have $\omega_{\Omega}\perp \cH^{\alpha}|_{\d\Omega}$ for $\alpha>1$ and $\omega_{\Omega}\ll \cH^{\alpha}|_{\d\Omega}$ if $\alpha<1$ \cite{Mak85}. This is a uniquely planar property, though: for $d\geq 2$, there are NTA topological spheres in $\bR^{d+1}$ called {\it Wolff snowflakes} for which either $\dim \omega_{\Omega}<d$ or $\dim \omega_{\Omega}>d$. In particular, we can have domains where $\omega_{\Omega}\ll \cH^{\alpha}$ on a set of positive harmonic measure for some $\alpha>d$ and $\omega_{\Omega}\perp \cH^{\alpha}$ for some $\alpha<d$. The $\bR^{3}$ case is due to Wolff \cite{Wolff91} and the result for higher dimensions is due to Lewis, Verchota, and Vogel \cite{LVV05}. A corollary of our main results, however, will show that, for NTA domains, mutual absolute continuity can only occur if $\alpha\leq d$. 

\begin{cori}
Let $\Omega\subseteq \bR^{d+1}$ be an NTA domain and $E\subseteq \bR^{d+1}$ such that $\omega_{\Omega}(E)>0$ and $\omega_{\Omega}|_{E}\ll \cH^{\alpha}|_{E}\ll \omega_{\Omega}|_{E}$. Then $\alpha\leq d$. 
\end{cori}

Our main result holds in more general circumstances. Firstly, the domain need not be NTA but just uniform, and the points in $E$ need to satisfy a nondegeneracy condition.

\begin{definition}
For $\Omega\subseteq \bR^{d+1}$ connected and $\beta,\delta\in (0,1)$, we say $\xi\in \d\Omega$ is {\it $(\beta,\delta)$-non-degenerate} if
\begin{equation}
\eta_{\delta}(\xi):=\limsup_{r\rightarrow 0}\eta_{\delta}(\xi,r)\leq \beta
\label{e:limsup}
\end{equation}
where 
\[
\eta_{\delta}(\xi,r):= \sup_{x\in \Omega \atop |x-\xi|= \delta r}\omega_{B(\xi,r)\cap \Omega}^{x}(\d B(\xi,r)\cap \Omega).\]
We will say $\xi$ is non-degenerate if it is $(\beta,\delta)$-non-degenerate for some $\beta,\delta>0$. 
\end{definition}

This may seem slightly messy, but it is satisfied at each point in $\d\Omega$ when $\Omega$ satisfies the capacity density condition, which we will define later. In particular, this includes NTA domains. 

Next, to establish the bound $\alpha\leq d$, we don't need mutual absolute continuity but just some control on the upper densities of harmonic measure. Recall that we define the {\it upper and lower $\alpha$-densities} for a measure $\mu$ as
\[ \theta^{\alpha,*}(\mu,\xi)=\limsup_{r\rightarrow 0} \frac{\mu(B(\xi,r))}{r^{s}}, \;\; \theta_{*}^{\alpha}(\mu,\xi)=\liminf_{r\rightarrow 0} \frac{\mu(B(\xi,r))}{r^{s}}.\]

We can now state the main result. 

\begin{thmi}
\label{t:thmi}
Let $d\geq 1$, and $\Omega\subseteq \bR^{d+1}$ be a uniform domain. Suppose there is $\alpha>0$, $x_{0}\in \Omega$, and a set $E\subseteq \d\Omega$ with $\omega_{\Omega}^{x_{0}}(E)>0$ such that each $\xi\in E$ is non-degenerate and
\begin{equation}
 0< \theta^{\alpha,*}(\omega_{\Omega}^{x_{0}},\xi)<\infty \;\; \mbox{ for }\;\; \xi\in E. 
 \label{e:densities}
 \end{equation}
Then $\alpha\leq d$.
\end{thmi}

Observe that if $E\subseteq \d\Omega$ is a set with $\cH^{\alpha}(E)<\infty$ and $\omega_{\Omega}^{x_{0}}\ll \cH^{\alpha}$ on $E$, then $\theta^{\alpha,*}(\omega_{\Omega}^{x_{0}},\xi)<\infty$ for $\omega_{\Omega}^{x_{0}}$-almost every $\xi\in E$, and so having finite densities is a weaker condition in this scenario. Indeed, for $\xi\in E$,
\begin{align*}
\theta^{\alpha,*}&(\omega_{\Omega}^{x_{0}},\xi)
 =\limsup_{r\rightarrow 0}\frac{\omega_{\Omega}^{x_{0}}(B(\xi,r))}{r^{\alpha}}\\
& = \limsup_{r\rightarrow 0} \frac{\omega_{\Omega}^{x_{0}}(B(\xi,r))}{\omega_{\Omega}^{x_{0}}(E\cap B(\xi,r))}\frac{\omega_{\Omega}^{x_{0}}(E\cap B(\xi,r))}{\cH^{\alpha}(E\cap B(\xi,r))}\frac{\cH^{\alpha}(E\cap B(\xi,r))}{r^{\alpha}}\\
\end{align*}

The first quotient converges to $1$ for $\omega_{\Omega}$-almost every $\xi\in E$ by \cite[Corollary 2.14]{Mattila}. The second quotient converges to a finite number for $\cH^{\alpha}$-almost every (and hence $\omega_{\Omega}$-almost every) $\xi\in E$ by \cite[Theorem 2.12]{Mattila}. Finally, by \cite[Theorem 6.2]{Mattila}, $\theta^{\alpha,*}(\cH^{d}|_{\d\Omega},\cdot)\in (0,\infty)$ and hence the limit of the third quotient has finite supremal limit for $\cH^{\alpha}$-almost every (and hence $\omega_{\Omega}$-almost every) $\xi\in E$. Thus, the above equalities give
\begin{equation}\label{e:twth}
\theta^{\alpha,*}(\omega_{\Omega}^{x_{0}},\cdot )\leq \frac{d \omega_{\Omega}|_{E}}{d\cH^{\alpha}|_{E}}\cdot  \theta^{\alpha,*}(\cH^{d}|_{\d\Omega},\cdot )
\;\;\; \omega_{\Omega}^{x_{0}}\mbox{-a.e. in $E$ if }\cH^{\alpha}(E)<\infty.
\end{equation}

In particular, this gives the following corollaries. 

\begin{corii}\label{c:corii}
Let $\Omega\subseteq \bR^{d+1}$ be a uniform domain, $\alpha>d$, and $E\subseteq \d\Omega$ be a set of non-degenerate points of finite $\cH^{\alpha}$-measure such that $\omega_{\Omega}^{x_{0}}|_{E}\ll \cH^{\alpha}|_{E}$. Then $\theta^{\alpha,*}(\omega_{\Omega}^{x_{0}},\xi)=0$ for $\omega_{\Omega}^{x_{0}}$-almost every $\xi\in E$.
\end{corii}

\begin{coriii}\label{c:coriii}
 Let $\Omega\subseteq \bR^{d+1}$ be a uniform domain, $E\subseteq \d\Omega$ a set of non-degenerate points, and $\alpha>d$. It is impossible for $\omega_{\Omega}^{x_{0}}|_{E} \ll \cH^{\alpha}|_{E} \ll \omega_{\Omega}^{x_{0}}|_{E}$ unless $\omega_{\Omega}^{x_{0}}(E)=0$.%
 \end{coriii}
 
 \CorollaryIII implies the conditions of \CorollaryII since $\cH^{\alpha}|_{E}\ll \omega_{\Omega}|_{E}$ implies $\cH^{\alpha}(E)<\infty$. \\

Mutual absolute continuity can in fact occur for $\alpha<d$. At the time of writing this manuscript, Alexander Volberg informed us that he constructed a uniform domain $\Omega\subseteq \bR^{d+1}$ satisfying the capacity density condition (so every point is non-degenerate by \Lemma{ancona} below) such that $\omega_{\Omega}\ll \cH^{\alpha}|_{\d\Omega}\ll \omega_{\Omega}$ for some $\alpha<d$. The construction is a modification of another example given by Bishop and Jones in the plane of a rectifiable set $E\subseteq \bR^{2}$ that contains a set of positive harmonic measure but zero Hausdorff $1$-measure \cite{BJ}.  \\

We can also bound $\alpha$ from below in certain circumstances.

\begin{thmii}
\label{t:thmii}
Let $\Omega\subseteq \bR^{d+1}$ be a uniform domain and let $E\subseteq \d\Omega$ have positive harmonic measure such that
\begin{equation}
 0< \theta^{\alpha}_{*}(\omega_{\Omega}^{x_{0}},\xi) <\infty \;\;\; \mbox{ for all }\;\;\; \xi\in E.
 \label{e:densities2}
 \end{equation}
Then $\alpha>d$. If for some $s>d-1$ we have for each $\xi\in E$
\begin{equation}
\liminf_{r\rightarrow 0} \frac{\cH^{s}_{\infty}(B(\xi,r)\cap \d\Omega)}{r^{s}}>0 \label{e:content}
\end{equation}
then $\alpha\geq s$.
\end{thmii}

Below we present a few simple corollaries of the main results. 
\begin{coriv}
If $\Omega$ is an NTA domain satisfying 
\[ 0< \theta^{\alpha}_{*}(\omega_{\Omega}^{x_{0}},\xi) \leq  \theta^{\alpha,*}(\omega_{\Omega}^{x_{0}},\xi) <\infty\]
for $\xi$ in a set of positive harmonic measure,  then $\alpha=d$. 
\end{coriv}
Indeed, it is not difficult to show that NTA domains satisfy \eqn{content} with $s=d$. If $B$ is a ball centered on $\d\Omega$ of radius $r_{B}$ and $B_{0}$ is a ball of radius $r_{B_{0}}$  comparable to $r_{B}$ so that $2B_{0}\subseteq \Omega\cap B$, then the existence of the exterior corkscrew ball implies that the radial projection of $\d \Omega\cap B$ onto $\d B_{0}$ has measure at least a constant times $r_{B_{0}}^{d}$ (and hence at least a constant times $r_{B}^{d}$). Since the radial projection onto $\d B_{0}$ is Lipschitz on $(2B_{0})^{c}$, this implies \eqn{content} with $s=d$. Also, any point satisfying \eqn{content} is non-degenerate by \Lemma{bourgain} below. Thus, the corollary follows from \TheoremI and \TheoremII.

This corollary is particularly interesting in the context of Wolff snowflakes. Recall that if $\mu$ is a Borel probability measure in $\mathbb{R}^{d+1}$, we define its \textit{lower and upper pointwise dimensions} at the point $x \in \supp \mu$ to be 
\[\underline d_\mu(x) = \liminf_{r \to 0} \frac{\log \mu(B(x,r))}{\log r} \;\; \mbox{ and } \;\; \overline d_\mu(x) = \limsup_{r \to 0} \frac{\log \mu(B(x,r))}{\log r}.\]
respectively. The common value $\underline d_\mu(x) =\overline d_\mu(x)=d_\mu(x)$, if it exists, we call it \textit{pointwise dimension} of $\mu$ at $x \in \supp \mu$. Wolff in fact constructs domains $\Omega\subseteq \bR^{d+1}$ where $d_{\omega_{\Omega}}<d$ $\omega_{\Omega}$-almost everywhere or $d_{\omega_{\Omega}}>d$ $\omega_{\Omega}$-almost everywhere. Note that if the upper and lower $\alpha$-densities are finite and positive at a point, this implies the pointwise dimension at that point is $\alpha$ as well. In other words, $\omega_{\Omega}$ having pointwise dimension $\alpha$ at $\xi$ means that for all $\ve>0$, $r^{\alpha+\ve}<\omega_{\Omega}(B(\xi,r))<r^{\alpha-\ve}$ for $r>0$ small enough, while having positive lower density means that $cr^{\alpha}<\omega_{\Omega}(B(\xi,r))<Cr^{\alpha}$ for $r$ small and some constants $c,C>0$. Thus, our results show that while the pointwise dimensions can be noninteger for these Wolff domains, the upper and lower $\alpha$-densities cannot be finite and positive on a set of positive measure.

%


To prove \TheoremI, we will rely heavily on the tangent measures of Preiss \cite{Preiss87}. Recall that if $\mu$ is a Radon measure and $x\in \supp \mu$, then the {\it tangent measures of $\mu$ at $x$}, denoted $\Tan(\mu,x)$, is the set of measures $\nu$ that are weak limits of the form $\nu =\lim_{j\rightarrow\infty}c_{j} T_{x,r_{j}\#}\mu$, where $c_{j}\geq 0$, $r_{j}\downarrow 0$, and 
\[ T_{x,r}(y)= \frac{y-x}{r}.\] 

Tangent measures have been employed to study the relationship between harmonic measure and the geometry of the boundary in several papers, see for example \cite{Bad11}, \cite{KPT09}, \cite{KT99}, and  \cite{KT06}. Other results which do not use tangent measures but employ more quantitative techniques modelled after tangent measure methods include \cite{PTT09}.

 The following is the main lemma we employ, whose proof takes up most of the paper and may be of independent interest.

\begin{lemi}\label{l:lemi}
Let $\Omega\subseteq \bR^{d+1}$ be a uniform domain, $d\geq 1$, and $x_{0}\in \Omega$. Fix  $\delta\in (0,1)$, let $E\subseteq \d\Omega$ be the set of $(\beta,\delta)$-non-degenerate points, and suppose it has positive $\omega_{\Omega}^{x_{0}}$-measure. Then for $\omega_{\Omega}^{x_{0}}$-almost every $\xi_{0}\in E$, $\Tan(\omega_{\Omega}^{x_{0}},\xi_{0})\neq\emptyset$. Moreover, if we have a tangent measure $\mu$ that is the weak limit of $T_{\xi_{0},r_{j}\#}\omega_{\Omega}^{x_{0}}/\omega_{\Omega}^{x_{0}}(B(\xi_{0},r_{j}))$, then we may pass to a subsequence such that the following hold.
\begin{enumerate}
\item $\supp \mu$ is the boundary of a $C'$-uniform domain $\tilde{\Omega}$ that is $\Delta$-regular (see Definition \ref{d:deltauniform} below), where $C'$ depends on $C$ and $d$, and the $\Delta$-regularity data depend additionally on $\delta$ and $\beta$. 
\item There is a uniform subdomain $\Omega^{*}$ dense in $\tilde{\Omega}$ such that for all $x\in \Omega^{*}$, if $\Omega_{j}:=T_{\xi_{0},r_{j}}(\Omega)$, then $x\in \Omega_{j}$ for all sufficiently large $j$. 
\item Let $\omega_{j}:=\omega_{\Omega_{j}}$. For $x\in \Omega^{*}$, $\omega_{j}^{x}$ converges weakly to $\omega_{\tilde{\Omega}}^{x}$. 
\item For continuous functions $f$ vanishing at infinity, the harmonic functions $\int f d\omega_{j}$ converge to $\int f d \omega_{\tilde{\Omega}}$ uniformly on compact subsets of $\Omega^{*}$. 
\item If  \eqn{content} holds, then there exists $c'>0$ depending on $c$ and $d$ so that $\cH_{\infty}^{s}(B(\xi,r)\cap \d\tilde{\Omega})\geq c'r^{s}$ for all $\xi\in \d\tilde{\Omega}$ and $r>0$.
\item Finally, there is $C_{0}$ depending on $d$ and $C$ so that if $B'\subseteq B=B(\xi,r)$  are balls centered on $\d\tilde{\Omega}$ and $B(x,\frac{r}{C'})\subseteq B\cap \Omega$, then
\begin{equation}
C_{0}^{-1} \frac{\mu(B')}{\mu(B)} \leq \frac{\omega_{\tilde{\Omega}}^{x}(B')}{\omega_{\tilde{\Omega}}^{x}(B)} \leq C_{0} \frac{\mu(B')}{\mu(B)}.
\label{e:mainlemma}
\end{equation}
\end{enumerate}
\end{lemi}

Similar results were shown by Kenig and Toro in \cite{KT99} in the case of NTA domains and \cite{KT06} in the case of two-sided NTA domains. For example, Lemma  3.8 in \cite{KT99} show the above for result for NTA domains, and the tangent measure $\mu$ is what they call the {\it tangent measure at $\infty$} for $\tilde{\Omega}$. The inequality \eqn{mainlemma} isn't stated there but follows from their work. See also Lemma 4.2 in \cite{KT06}). What is special about the above lemma, however, is that it works for more general domains, and secondly, that we can fix a point in the limiting domain $\tilde{\Omega}$ and the scaled harmonic measures $\omega_{j}^{x}$ will converge to the corresponding harmonic measure in $\tilde{\Omega}$. In a recent paper with Xavier Tolsa, we also obtain slightly weaker versions of the blow up results of Kenig and Toro that held for $\Delta$-regular domains without assuming uniformity (see \cite{AMT16}); however $\Delta$-regularity is much stronger than the assumptions in \LemmaI, and we did not obtain \eqn{mainlemma} in the purely $\Delta$-regular (non-uniform) setting.

Note that $\tilde{\Omega}$ in \LemmaI is a uniform $\Delta$-regular domain, and thus a uniform domain satisfying the CDC (if $d\geq 2$). This latter set satisfies many useful properties (such as harmonic measure being doubling, see \cite{AH08}) that the original domain $\Omega$ may not have enjoyed originally.\\

There are several possible venues for improvement and inquiry. Firstly, can we relax the uniformity and nondegeneracy conditions? These are used in quite crucial ways in the proof. Secondly, we note that in Volberg's example, $\theta^{\alpha}_{*}(\cH^{d}|_{\d\Omega},\cdot)=0$ $\cH^{\alpha}$-almost everywhere, and hence $\theta^{\alpha}_{*}(\omega_{\Omega}|,\cdot)$ vanishes $\omega_{\Omega}$-almost everywhere, and we don't know if $\alpha=d$ otherwise.\\

The authors would like to thank Xavier Tolsa for pushing us to eliminate a strong assumption from the main result, and Alexander Volberg for his enlightening discussions and comments on the manuscript.

\section{Preliminaries}

\subsection{Notation}

We will work entirely in in $\bR^{d+1}$ with $d\geq 1$.  We write $a\lesssim b$ if there is $C>0$ so that $a\leq Cb$ and $a\lesssim_{t} b$ if $C$ depends on the parameter $t$. We write $a\sim b$ to mean $a\lesssim b\lesssim a$ and define $a\sim_{t}b$ similarly. To simplify notation, we will implicitly assume that all implied constants depend on $d$

The open ball centered at $\xi\in \bR^{d+1}$ of radius $r>0$ will be denoted $B(\xi,r)$, and in particular, we will write $\bB:=B(0,1)$. If $x\in \bR^{d}$, we will denote the $d$-dimensional ball by $B(\xi,r)=B(\xi,r)\cap\bR^{d}$. If $B$ is a ball, its radius will be denoted $r_{B}$. If $\Omega\subseteq \bR^{d+1}$ is a domain, we will write $\Omega^{\ext}=(\cnj{\Omega})^{c}$.

If $\Omega$ is uniform and $B$ is centered on the boundary, we will write $x_{B}$ for a point such that $B(x_{B},r_{B}/C)\subseteq B\cap \Omega$ (or, to permit us some flexibility, any point $x_{B}$ satisfying this with constant $C$ comparable to the original constant in the definition of uniform domains). 

For sets $A,B\subseteq \bR^{d+1}$, we let 
\[\dist(A,B)=\inf\{|x-y|:x\in A,y\in B\}, \;\; \dist(x,A)=\dist(\{x\},A),\]
and 
\[\diam A=\sup\{|x-y|:x,y\in A\}.\]
For $A\subseteq \bR^{d+1}$, $\alpha>0$, and $\delta\in (0,\infty]$, define
\[\cH^{\alpha}_{\delta}(A)=\inf\ck{\sum r_{i}^{\alpha}: A\subseteq \bigcup B(x_{i},r_{i}),x_{i}\in\bR^{d+1}, \;\; r_{i}<\delta }.\]
We define the {\it $\alpha$-dimensional Hausdorff measure} as
\[\cH^{\alpha}(A)=\lim_{\delta\downarrow 0}\cH^{\alpha}_{\delta}(A),\]
the {\it $d$-dimensional Hausdorff content} as $\cH^{\alpha}_{\infty}(A)$, and the {\it Hausdorff dimension of $A$} as $\dim A=\inf\{\alpha:\cH^{\alpha}(A)=0\}$. See \cite[Chapter 4]{Mattila} for more information about Hausdorff measure.

\subsection{Regularity of harmonic functions}

Here we collect some lemmas about harmonic measure.

\begin{definition}
A domain $\Omega\subseteq \bR^{d+1}$ satisfies the {\it Harnack chain condition} if there is $C>0$ so that for all $\Lambda$ there is $N(\Lambda)$ such that for all $\ve>0$ and $x,y\in \Omega$ with $\dist(\{x,y\},\d\Omega)\geq \ve$ and $|x-y|\leq \Lambda \ve$, there is a chain of balls $B_{1},...,B_{N}\subseteq \Omega$ with 
\begin{enumerate}
\item $N\leq N(\Lambda)$,
\item $r_{B_{i}}/C\leq \dist(B_{i},\d\Omega)\leq Cr_{B_{i}} $ for $i=1,...,N$,
\item $B_{i}\cap B_{i+1}\neq\emptyset$ for $i=1,...,N-1$, and
\item $x\in B_{1}$ and $y\in B_{N}$.
\end{enumerate}
In particular, if $u$ is a positive harmonic function on $\Omega$, then by repeated use of Harnack's inequality on each $B_{i}$,
\begin{equation}
u(x)\sim_{\Lambda} u(y) \;\;\; \mbox{ if } \;\;\; \frac{|x-y|}{\dist(\{x,y\},\d\Omega)}\leq \Lambda.
\label{e:harnack}
\end{equation}
\label{d:harnack-uniform}
\end{definition}

\begin{lemma}[Theorem 2.15 in \cite{AHMNT}]
A domain is uniform if and only if it satisfies the interior corkscrew and Harnack chain conditions quantitatively.
\label{l:ahmnt}
\end{lemma}

The way nondegeneracy will manifest in our proof is the following lemma.

\begin{lemma}
\label{l:holder}
Suppose $\Omega\subseteq \bR^{d+1}$, $\delta\in (0,1)$, $\xi\in \d\Omega$ and $\omega_{B\cap \Omega}^{x}(\d B(\xi,r)\cap \Omega)  \leq \beta<1$ for $x\in  \d B(\xi,\delta r)\cap \Omega$ and $r\in (0,R)$. Then there is $\alpha=\alpha(\beta,d)$ so that for all $r\in (0,R)$
\begin{equation}\label{e:wholder}
 \omega_{\Omega}^{x}(\cnj{B(\xi,r)}^{c})\lesssim_{\beta,\delta} \ps{\frac{|x-\xi|}{r}}^{\alpha} \mbox{ for } x\in \Omega\cap B(\xi,r).
 \end{equation}
In particular, $\xi$ is a regular point for $\d\Omega$. 
\end{lemma}

\begin{proof}
Let $B=B(\xi,r)$, $r<R$, and $\phi$ be a continuous function such that $\one_{B}\leq\phi \leq \one_{2B}$ and let $\psi=1-\phi$. Let $u_{\psi}=\int \psi d\omega_{\Omega}$. Then by the maximum principle, for  $x\in \delta B$, 
\[u_{\psi}(x)\leq \omega_{\Omega}^{x}(B^{c})\leq\omega_{\Omega\cap B}^{x}(\d B\cap \Omega)\leq \beta<1.\]
Thus, again by the maximum principle, for $x\in \delta^{2}B\cap\Omega$,
\[u_{\psi}(x)\leq \beta \omega_{\Omega\cap \delta B}^{x}(\d (\delta B)\cap \Omega)\leq \beta^{2}\]
and inductively, we have
\[ u_{\psi}(x)\leq \beta^{j} \mbox{ for }x\in \delta^{j} B\cap\Omega\mbox{ and } j\geq 0.\]
Thus there is $\alpha=\alpha(\beta,\delta)>0$ such that
\[ u_{\psi}(x)\lesssim_{\beta,\delta} \ps{\frac{|x-\xi|}{r}}^{\alpha}\mbox{ for }x\in \delta B\cap \Omega.\]
Having $\phi$ decrease pointwise to $\one_{\cnj{\bB}}$, we have 
\[ \omega_{\Omega}^{x}(\cnj{B}^{c})\lesssim_{\beta,\delta} \ps{\frac{|x-\xi|}{r}}^{\alpha} \mbox{ for }x\in \delta B\cap \Omega.\]
\end{proof}

%
%
%
%

\begin{definition} \cite{Anc86} For a domain $\Omega\subseteq \bR^{d+1}$ and a ball $B$ centered on $\d\Omega$, and $x\in B\cap \Omega$. We say that $\Omega$ is {\it uniformly $\Delta$-regular} if there are $\delta\in (0,1)$ and $R_{\Delta}\in (0,\infty]$ so that 
\begin{equation}
\sup_{\xi\in \d\Omega}\sup_{r<R_{\Delta}}\eta_{\delta}(\xi,r)<1.\
\label{e:duniform}
\end{equation}
\label{d:deltauniform}
\end{definition}

The original definition is given with $\delta=\frac{1}{2}$, but it is not difficult (but still tedious) to show that for uniform domains if the definition holds for one value $\delta$ then it holds for any $\delta\in (0,1)$. 
%
%
%
%
%
%
%

\begin{definition} \cite{Aik01} Let $d\geq 2$ and let $\Cap$ denote the Newtonian capacity. A domain $\Omega\subseteq \bR^{d+1}$ satisfies the {\it capacity density condition} (or CDC) if there is $R_{\Omega}>0$ so that $\Cap(B\backslash \Omega)\gec r_{B}^{d-1}$ for any ball $B$ centered on $\d\Omega$ of radius $r_{B}\in (0,R_{\Omega})$. 
\label{d:cdc}
\end{definition}

The same definition works in the plane with the logarithmic capacity, but we will not use it here.

It was shown in \cite{Anc86}  that the CDC is equivalent to $\Delta$-regularity for $d\geq 2$. 

\begin{theorem}  \cite[Lemma 3]{Anc86} For $d\geq 2$, there is $c\geq 4$ so that if  $\Omega\subseteq\bR^{d+1}$ and $B$ is centered on $\d\Omega$, then $\Cap(B\backslash \Omega)\gec r_{B}^{d-1}$ if and only if there is $\beta\in (0,1)$ so that $\omega_{cB\cap \Omega}^{x}(\d (cB)\cap \Omega)\leq \beta$ on $\d (2B)\cap \Omega$. In particular, $\Omega$ is uniformly $\Delta$-regular if and only if it satisfies the CDC. 
\label{l:ancona}
\end{theorem}

\begin{lemma} \cite[Lemma 1]{Bou87} \cite[Lemma 3.4]{AHMMMTV15} 
\label{l:bourgain}
Let $d\geq 1$ and $\Omega\subseteq \bR^{d+1}$ be a domain, $\xi\in \d\Omega$, $r>0$, $B:=B(\xi,r)$, and suppose that $\rho:=\cH_{\infty}^{s}(\d\Omega\cap \delta B)/(\delta r)^{s}$ for some $s>d-1$. Then 
\begin{equation}
 \omega_{\Omega\cap B}^{x}(B)\gtrsim \rho\; \mbox{  for all }x\in \delta B\cap \Omega.
 \label{e:bourgain}
 \end{equation}
In particular, if $\xi$ satisfies \eqn{content}, then $\xi$ is a non-degenerate point. If $\cH^{s}(\d\Omega\cap \delta B)/(\delta r_{B})^{s}\gec 1$ for all balls $B$ centered on $\d\Omega$ of radius less than some $r_{0}$, then $\Omega$ is $\Delta$-regular.
\end{lemma}

For the $d=2$ case, see \cite{Bou87}; the general case is identical, but a proof is given in \cite{AHMMMTV15} as well.

Finally, we recall some lemmas from \cite{MT15}.

\begin{lemma}\cite[Theorem 1.3]{MT15}
Let $d\geq 1$, $\Omega$ be a $C$-uniform domain in $\bR^{d+1}$ and let $B$ be a ball centered at $\partial\Omega$.
Let $p_1,p_2\in\Omega$ be such that $\dist(p_i,B\cap \partial\Omega)\geq c_0^{-1}\,r_{B}$ for $i=1,2$.
Then, for all $E\subset B\cap\d\Omega$,
\begin{equation}
\frac{\omega_{\Omega}^{p_1}(E)}{\omega_{\Omega}^{p_1}(B)}\sim_{c_{0},C} \frac{\omega_{\Omega}^{p_2}(E)}{\omega_{\Omega}^{p_2}(B)}.
\label{e:w/w}
\end{equation}
\label{l:w/w}
\end{lemma}

\begin{lemma}\label{l:w<G}
\cite[Lemma 10.1]{MT15} Let $d\geq 1$ and let $\Omega\subsetneq \bR^{d+1}$ be a $C$-uniform domain  and $B$ a ball centered at $\partial\Omega$ with radius $r$.  Suppose that there exists a point $x_B \in \Omega$ so that the ball $B_0:=B(x_{B},r/M)$ satisfies $4B_0\subset \Omega\cap B$ for some $M>1$. Then, for  $0<r\leq r_\Omega=r_{\Omega}(d,C)$,
and $\tau>0$,
\begin{equation}\label{e:Green-upperbound}
 \omega_{\Omega}^{x}(B)\sim_{C,M,\tau,d} \omega_{\Omega}^{x_{B}}(B) G_{\Omega}(x,x_{B}) r^{d-1}\,\, \,\,\,\,\,\,\text{for all}\,\, x\in \Omega\backslash  (1+\tau)B.
 \end{equation}
Here, $r_\Omega=\infty$ if $\diam(\Omega)=\infty$.
\end{lemma}

\subsection{Tangent measures}

We recall some basic results. 

\begin{lemma} \cite[Theorem 14.3]{Mattila} Let $\mu$ be a Radon measure on $\bR^{n}$. If $a\in \bR^{n}$ and 
\[\limsup_{r\rightarrow 0} \frac{\mu(B(a,2r))}{\mu(B(a,r))}<\infty,\]
then every sequence $r_{i}\downarrow 0$ contains a subsequence such that $T_{a,r_{j}\#}\mu/\mu(B(a,r_{j}))$ converges to a measure $\nu\in \Tan(\mu,a)$. 
\label{l:tanexist}
\end{lemma}

\begin{lemma} \label{l:same-tangents}
\cite[Lemma 14.5]{Mattila}
Let $\mu$ be a Radon measure on $\bR^{n}$ and $A$ a  measurable set. Suppose $a\in \supp \mu$ is a point of density for $A$, meaning
\[\lim_{r\rightarrow 0} \frac{\mu(B(a,r)\backslash A)}{\mu(B(a,r))}=0.\]
If $c_{i} T_{a,r_{i}\#}\mu \rightarrow \nu\in \Tan(\mu,a)$, then so does $c_{i} T_{a,r_{i}\#}\mu|_{A}$. In particular, this holds for $\mu$ almost every $x\in A$. 
\end{lemma}

The above lemma is not stated as such in \cite{Mattila}, but it follows by an inspection of the proof (in particular the last two lines). The way we will use this lemma is the following.

\begin{corollary}
With the assumptions of the previous lemma, if $c_{i} T_{a,r_{i}\#}\mu \rightarrow \nu\in \Tan(\mu,a)$, then for all $\xi\in \supp \nu$ and $\rho>0$, $T_{a,r_{i}}(A)\cap B(\xi,\rho)\neq\emptyset$ for $j$ sufficiently large. In particular, we may find $\xi_{i}\in T_{a,r_{i}}(A)$ so that $\xi_{i}\rightarrow \xi$. 
\label{c:hitsball}
\end{corollary}


\begin{lemma} \cite[Lemma 14.7]{Mattila}
Let $\mu$ be a Radon measure in $\bR^{n}$, $s>0$, and let $A$ be the set of points $\xi$ in $\bR^{n}$ for which
\[ 0<a\leq \theta_{*}^{s}(\mu,\xi)\leq \theta^{s,*}(\mu,\xi)\leq b<\infty.\]
Then for almost every $\xi\in A$ and every $\nu\in \Tan(\mu,\xi)$,
\begin{equation}
 ar^{s}\leq \nu(B(x,r))\leq br^{s}\;\; \mbox{ for }x\in \supp \nu, 0<r<\infty.
 \label{e:regular}
 \end{equation}
\label{l:regular}
\end{lemma}

A measure satisfying \eqn{regular} for some $a,b>0$ is called {\it Ahlfors $s$-regular}. We will need a slightly different version of this result suggested by Xavier Tolsa.

\begin{lemma}
Let $\mu$ be a Radon measure in $\bR^{n}$, $s>0$. Then for $\mu$ almost every $\xi_{0}\in S=\{\xi\in \bR^{n}:0<\theta^{s,*}(\mu,\xi)<\infty\}$, there is $\nu\in \Tan(\mu,\xi_{0})$ and $r_{j}$ such that 
\begin{equation}
\mu_{j}=\frac{T_{\xi_{0},r_{j}\#}\mu }{ \mu(B(\xi_{0},r_{j}))}\rightarrow \nu
\label{e:upperlim1}
\end{equation}
and
\begin{equation}
\nu(B(x,r))\leq  r^{s} \mbox{ for all }x\in \supp \nu, r>0.
\label{e:upperlim2}
\end{equation}
\label{l:upperlim}
\end{lemma}

\begin{proof}
Note that the function $\theta^{*,s}(\mu,x)$ is a Borel function (see \cite[Chapter 6, Exercise 3]{Mattila}). By Ergorov's theorem, for $k\in \bN$, we may find a set $S_{k} \subseteq S\cap B(0,k )$ so that $\mu(S\cap B(0,k) \backslash S_{k})<k^{-1}$ and $\theta^{*,s}(\mu,\cdot)$ is continuous on $S_{k}$. 

For integers $k,\ell,m$,  let 
\begin{equation}
S_{k,\ell,m}=\{\xi\in S_{k}: \mu(B(\xi,r))/r^{s}\leq (1+\ell^{-1})\theta^{*,s}(\mu,\xi) \mbox{ for }r\in (0,m^{-1})\}
\label{e:Slm}
\end{equation}
and let $S_{k,\ell,m}^{*}$ be the points of density for this set. Then for each $m\in \bN$, almost all of $S_{k}$ is in $\bigcup_{m}S_{k,\ell,m}^{*}$, and hence almost all of $S_{k} $ is in 
\[ S_{k}^{*}:=\bigcap_{\ell}\bigcup_{m}S_{k,\ell,m}^{*}\subseteq S_{k}\]
Thus, almost all of $S$ is in $S^{*}=\bigcup S_{k}^{*}$.

Let $\xi\in S^{*}$.  Pick $r_{j}\downarrow 0$ such that 
\begin{equation}
\frac{\mu(B(\xi,r_{j}))}{r_{j}^{s}}\rightarrow \theta^{*,s}(\mu,\xi)\in (0,\infty).
\label{e:mu/rtheta}
\end{equation}
Let $\mu_{j}=T_{\xi,r_{j}\#} \mu/\mu(B(\xi,r_{j}))$. We first claim we can pick a subsequence so that this converges weakly to a nonzero measure $\nu\in \Tan(\mu,\xi)$. To see this, observe that since $\theta^{*,s}(\mu,\xi)<\infty$, for any $R$ we have that 
\begin{multline*}
\limsup_{j}\mu_{j}(B(0,R))
=\limsup_{j}\frac{\mu(B(\xi,Rr_{j}))}{\mu(B(\xi,r_{j}))}\\
=\limsup_{j} \frac{\mu(B(\xi,Rr_{j}))}{(Rr_{j})^{s}} \cdot \lim_{j\rightarrow\infty}\frac{r_{j}^{s}}{\mu(B(\xi,r_{j}))} R^{s}
\stackrel{\eqn{mu/rtheta}}{\leq} R^{s}.
\end{multline*}
Since this holds for all $R$, we can use a diagonalization argument to pick a subsequence so that $\mu_{j}$ converges weakly on all of $\bR^{n}$ to a finite measure $\nu\in \Tan(\mu,\xi)$.

Let $x\in \supp\nu$ and $r>0$. Let $\ell\in \bN$, so $\xi\in S_{k,\ell,m}^{*}$ for some $k,m$. By \Corollary{hitsball}, we may find $\xi_{j}\in T_{\xi,r_{j}}(S_{k,\ell,m})$ so that $\xi_{j}\rightarrow x$. Let $\zeta_{j}=  T_{\xi,r_{j}}^{-1}(\xi_{j})$. Then $\theta^{*,s}(\mu,\zeta_{j})\rightarrow \theta^{*,s}(\mu,\xi)$ since $\zeta_{j},\xi\in S_{k,\ell,m}$, $\zeta_{j}\rightarrow \xi$, and $\theta^{*,s}(\mu,\cdot)$ is continuous on $S_{k,\ell,m}$. Hence,
\begin{align*}
\nu(B(x,r))
& \leq \liminf_{j\rightarrow\infty}\mu_{j}(B(x,r))
\leq \liminf_{j\rightarrow\infty}\mu_{j}(B(\xi_{j},|\xi_{j}-x|+r))\\
& =\liminf_{j\rightarrow\infty}\frac{\mu(B(\zeta_{j},r_{j}(|\xi_{j}-x|+r)))}{\mu(B(\xi,r_{j}))}\\
& \stackrel{\eqn{Slm}}{\leq} \liminf_{j\rightarrow\infty}\frac{(1+\ell^{-1})\theta^{*,s}(\mu,\zeta_{j}) r_{j}^{s}(|\xi_{j}-x|+r)^{s}}{r_{j}^{s}}\frac{r_{j}^{s}}{\mu(B(\xi,r_{j}))}\\
&  \stackrel{\eqn{mu/rtheta}}{=}(1+\ell^{-1})\theta^{s,*}(\mu,\xi) (0+r)^{s}\cdot \theta^{s,*}(\mu,\xi)^{-1}
=(1+\ell^{-1}) r^{s}
\end{align*}

Since this holds for all $\ell\in \bN$, the lemma follows.

\end{proof}

\begin{lemma}
Let $\mu$ be a Radon measure in $\bR^{n}$, $s>0$. Let
\[
S=\{\xi\in \bR^{n}:0<\theta^{s}_{*}(\mu,\xi)<\infty, \limsup_{r\rightarrow 0}\mu(B(\xi,2r))/\mu(B(\xi,r))<\infty\}.\] 
Then for almost every $\xi_{0}\in S$ there is $\nu\in \Tan(\mu,\xi_{0})$ and $r_{j}$ such that 
\begin{equation}
\mu_{j}=\frac{T_{\xi_{0},r_{j}\#}\mu }{ \mu(B(\xi_{0},r_{j}))}\rightarrow \nu
\label{e:upperlim12}
\end{equation}
and
\begin{equation}
\nu(B(x,r))\geq  r^{s} \mbox{ for all }x\in \supp \nu, r>0.
\label{e:upperlim22}
\end{equation}
\label{l:upperlim2}
\end{lemma}

\begin{proof}
Again, $\theta^{s}_{*}(\mu,x)$ is a Borel function (see \cite[Chapter 6, Exercise 3]{Mattila}), so Ergorov's theorem implies for $k\in \bN$, we may find a set $S_{k} \subseteq S\cap B(0,k )$ so that $\mu(B(S\cap B(0,k))\backslash S')<k^{-1}$ and $\theta^{s}_{*}(\mu,\cdot)$ is continuous on $S_{k}$. 

For integers $k,\ell,m$,  let 
\begin{equation}
S_{k,\ell,m}=\{\xi\in S_{k}: \mu(B(\xi,r))/r^{s}\geq (1-\ell^{-1})\theta^{s_{*}}(\mu,\xi) \mbox{ for }r\in (0,m^{-1})\}
\label{e:Slm2}
\end{equation}
and let $S_{k,\ell,m}^{*}$ be the points of density for this set. Then for each $m\in \bN$, almost all of $S_{k}$ is in $\bigcup_{m}S_{k,\ell,m}^{*}$, and hence almost all of $S_{k} $ is in 
\[ S_{k}^{*}:=\bigcap_{\ell}\bigcup_{m}S_{k,\ell,m}^{*}\subseteq S_{k}\]
Thus, almost all of $S$ is in $S^{*}=\bigcup S_{k}^{*}$.

Let $\xi\in S^{*}$.  Pick $r_{j}\downarrow  0$ such that 
\begin{equation}
\frac{\mu(B(\xi,r_{j}))}{r_{j}^{\alpha}}\rightarrow \theta^{d}_{*}(\mu,\xi)\in (0,\infty).
\label{e:mu/rtheta2}
\end{equation}
By \Lemma{tanexist} and the definition of $S$, we can pass to a subsequence so that $\mu_{j}=T_{\xi,r_{j}\#} \mu/\mu(B(\xi,r_{j}))$ converges weakly to a nonzero measure $\nu\in \Tan(\mu,\xi)$. Let $x\in \supp\nu$ and $r>0$. Let $\ell\in \bN$, so $\xi\in S_{k,\ell,m}^{*}$ for some $k,m$. By \Corollary{hitsball}, we may find $\xi_{j}\in T_{\xi,r_{j}}(S_{k,\ell,m})$ so that $\xi_{j}\rightarrow x$. Let $\zeta_{j}=  T_{\xi,r_{j}}^{-1}(\xi_{j})\in S_{k,\ell,m}$. Then
\begin{align*}
\limsup_{j\rightarrow\infty} & \mu_{j}(B(x,r))
 \geq \limsup_{j\rightarrow\infty}\mu_{j}(B(\xi_{j},r-|\xi_{j}-x|))\\
& =\limsup_{j\rightarrow\infty}\frac{\mu(B(\zeta_{j},r_{j}(r-|\xi_{j}-x|)))}{\mu(B(\xi,r_{j}))}\\
& \stackrel{\eqn{Slm2}}{\geq} \limsup_{j\rightarrow\infty}\frac{(1-\ell^{-1})\theta^{s}_{*}(\mu,\zeta_{j}) r_{j}^{s}(r-|\xi_{j}-x|)^{s}}{r_{j}^{s}}\frac{r_{j}^{s}}{\mu(B(\xi,r_{j}))}\\
&  \stackrel{\eqn{mu/rtheta2}}{=}(1-\ell^{-1})\theta^{s}_{*}(\mu,\xi) (r-0)^{s}\cdot \theta^{s}_{*}(\mu,\xi)^{-1}
=(1-\ell^{-1}) r^{s}
\end{align*}

Since this holds for all $\ell\in \bN$, for all $\ve>0$, we have
\[
\nu(B(x,r+\ve))
\geq \nu(\cnj{B(x,r)})\geq \limsup_{j\rightarrow\infty}\mu_{j}(\cnj{B(x,r)})
\geq  \mu_{j}(B(x,r)) \geq  r^{s}.\]
Thus $\nu(\cnj{B(x,r)})\geq r^{s}$ for all $x\in \supp \nu$ and $r>0$. This easily implies $\nu({B(x,r)})\geq r^{s}$  for all $x\in \supp \nu$ and $r>0$.

\end{proof}

\begin{lemma}
Suppose $\Omega\subseteq \bR^{d+1}$ is a uniform domain, $x_{0}\in \Omega$, $\xi_{0}\in \d\Omega$ is non-degenerate, and $r_{j}\rightarrow 0$. Then there is a subsequence such that $\mu_{j}:=T_{\xi_{0},r_{j}\#}\omega_{\Omega}^{x_{0}}/\omega_{\Omega}^{x_{0}}(B(\xi_{0},r_{j}))$ converges weakly to a measure $\nu\in \Tan(\omega_{\Omega}^{x_{0}},\xi_{0})$. 
\label{l:tanonempty}
\end{lemma}

\begin{proof}
Suppose $\xi$ is $(\beta,\delta)$-non-degenerate for some $\beta,\delta\in(0,1)$. By \cite[Theorem 14.3]{Mattila}, we need only show that 
\begin{equation}
\limsup_{r\rightarrow 0} \frac{\omega_{\Omega}^{x_{0}}(B(\xi_{0},2r))}{\omega_{\Omega}^{x_{0}}(B(\xi_{0},r))}<\infty.
\label{e:dublin}
\end{equation}
Note that for $r>0$ small enough, $x_{0}\not\in B(\xi_{0}, 4r)$, and so we may apply \Lemma{w/w} to get
\begin{align}
\frac{\omega_{\Omega}^{x_{0}}(B(\xi_{0},2r))}{\omega_{\Omega}^{x_{0}}(B(\xi_{0},r))}
& \stackrel{\eqn{w/w}}{\sim}_{C}  
\omega_{\Omega}^{x_{B(\xi_{0},2r)}}(B(\xi_{0},r))^{-1}
\stackrel{\eqn{harnack}}{\sim}_{C,\delta}
\omega_{\Omega}^{x_{B(\xi_{0},\delta r)}}(B(\xi_{0},r))^{-1} \notag \\
& \leq \omega_{\Omega\cap B(\xi_{0},r)}^{x_{B(\xi_{0},\delta r)}}(B(\xi_{0},r)\cap \d\Omega)^{-1}
\label{e:dublin2}
\end{align}
Thus, since $\xi_{0}$ is non-degenerate, we now have \eqn{dublin} since
\[
\limsup_{r\rightarrow 0} \frac{\omega_{\Omega}^{x_{0}}(B(\xi_{0},2r))}{\omega_{\Omega}^{x_{0}}(B(\xi_{0},r))}
\stackrel{\eqn{dublin2}}{\lec }
\limsup_{r\rightarrow 0} \omega_{\Omega\cap B(\xi_{0},r)}^{x_{B(\xi_{0},\delta r)}}(B(\xi_{0},r)\cap \d\Omega)^{-1}
<\infty.\]

\end{proof}

\section{Proof of \LemmaI}

Let $\Omega\subseteq \bR^{d+1}$ be a uniform domain and $\xi_{0}\in \d\Omega$ and $r_{j}\rightarrow 0$. Let $T_{j}=T_{\xi_{0},r_{j}}$ and suppose $\mu_{j}=T_{j\#}\omega_{\Omega}^{x_{0}}/\omega_{\Omega}^{x_{0}}(B(\xi_{0},r_{j}))$ converges weakly to a measure $\mu$. Let $\Omega_{j}=T_{j}(\Omega)$. Pass to a subsequence so that $\d\Omega_{j}\cap \cnj{B(0,n)}$ converges in the Hausdorff metric to a compact set $\Sigma_{n}$. Note that $\Sigma_{n}\subseteq \Sigma_{n+1}$, otherwise there is $x\in \Sigma_{n}$ of distance $r>0$ from $\Sigma_{n+1}$. For large $j$ there is $\xi_{j}\in \d\Omega_{j}\cap \cnj{B(0,n)}\cap B(x,r/2)$. Since $\xi_{j}\in \d\Omega_{j}\cap \cnj{B(0,n+1)}$, $r/2\leq \dist(\xi_{j},\Sigma_{n+1})\rightarrow 0$ as $j\rightarrow \infty$, a contradiction. Let 
\[\Sigma=\bigcup \Sigma_{n}.\] 

\begin{lemma}
We have $\supp \mu \subseteq \Sigma$. 
\end{lemma}

\begin{proof}
Suppose $B$ is a ball centered on $\supp \mu$. Then $\mu_{j}(B)>0$ for all $j$ large, hence $B\cap \d\Omega_{j}\neq 0$ for all $j$ large since $\supp \mu_{j}=\d\Omega_{j}$. If $\xi_{j}\in \d\Omega_{j}\cap B$, then there is a subsequence converging to a point $\xi\in \Sigma\cap \cnj{B}$. Thus, $\Sigma$ intersects the closure of any ball centered on $\supp \mu$, which implies $\supp \mu\subseteq \Sigma$. 
\end{proof}

\begin{lemma}
By passing to a subsequence, we may assume that for all $\xi\in \Sigma$ and $r>0$ there is a ball $B$ of radius $\frac{r}{8C}$ so that $2B\subseteq B(\xi,r)\cap \Omega_{j}$ for all large $j$.
\label{l:cork-in-limit}
\end{lemma}

\begin{proof}

Let $A$ be a countably dense set in $\Sigma$. Let $\xi\in A$ and $r\in \bQ\cap (0,\infty)$. Then there is $\xi_{j}\in \Omega_{j}\cap B(\xi,r/2)$ for $j$ large enough. Since $\Omega_{j}$ satisfies the $C$-interior corkscrew condition, there is a ball
$B(x_{j},\frac{r}{2C})\subseteq B(\xi_{j},r/2)\cap \Omega_{j}$. By passing to a subsequence, we can find $x_{\xi,r}$ so that for large $j$
\[
B\ps{x_{\xi,r},\frac{2r}{5C}}\subseteq \Omega_{j}\cap B(\xi_{j},r/2)\subseteq \Omega_{j}\cap B(\xi,r).\]
By a diagonalization argument, we can assure that this holds for all $(\xi,r)\in A\times (0,r)\cap \bQ$ for sufficiently large $j$. By the density of $A$ and $(0,r)\cap \bQ$, it follows that for all $\xi\in \Sigma$ and $r>0$, there is $x_{\xi,r}$ so that $B(x_{\xi,r},\frac{r}{4C})\subseteq B(\xi,r)\cap\Omega_{j}$ for all $j$ large. By taking $B=B(x_{\xi,r},\frac{r}{8C})$, this proves the lemma.

\end{proof}

For all $x\in \bQ^{d+1}\backslash \Sigma$, let $r_{x}=\dist(x,\Sigma)$ so that $B(x,r_{x})\subseteq \Sigma^{c}$ and $B_{x}:=B(x,r_{x}/2)\subseteq (\d\Omega_{j})^{c}$ for all sufficiently large $j$. By a diagonalization argument, we may pass to a subsequence such that for each $x\in \bQ^{d+1}\backslash \Sigma$, $B_{x}\subseteq \Omega_{j}$ for all but finitely many $j$ or $B_{x}\subseteq \Omega_{j}^{\ext}$ for all but finitely many $j$. Let 
\begin{equation}
\Omega^{*}=\bigcup\{B_{x}: B_{x}\subseteq \Omega_{j} \mbox{ for all but finitely many }j\}.\
\label{e:omegastar}
\end{equation}

\begin{lemma}
$\Omega^{*}$ is a $C'$-uniform domain with constant depending on $C$, $\d\Omega^{*}=\Sigma$, and $\Omega^{*}$ satisfies the $8C$-interior corkscrew property.
\label{l:omegastaruniform}
\end{lemma}

\begin{proof}
By a covering argument, it follows that any ball $B$ with $B\subseteq \Omega_{j}$ for all $j$ large satisfies $B\subseteq \Omega^{*}$. By \Lemma{cork-in-limit}, for every $\xi\in \Sigma$ and $r>0$, there is a ball $B\subseteq B(\xi,r)\cap \Omega_{j}$ of radius $\frac{r}{8C}$ for $j$ large, and hence $B\subseteq \Omega^{*}$. Since this holds for all $\xi\in \Sigma$ and $r>0$, this implies $\Sigma\subseteq \cnj{\Omega^{*}}$. By construction, however, $\Omega^{*}\subseteq \Sigma^{c}$, and so $\Sigma\subseteq \d\Omega^{*}$. Hence, we have shown that $\Omega^{*}$ satisfies the interior corkscrew property with constant $8C$, and in particular, has nonempty interior.

%

Now we focus on uniformity, but to prove this, we will need the following theorem.

\begin{theorem} \label{t:martin} \cite[Theorem 5.1]{Mar85}
Let $\Omega$ be a uniform domain. Then there is a constant $L$, depending only on the uniformity constant for $\Omega$, such that for each pair of points $x,y\in \Omega$ there is an $L$-bi-Lipschitz embedding $f:\cnj{B(0,|x-y|)}\rightarrow \Omega$ such that $\{x,y\}\subseteq f(\cnj{B(0,|x-y|)})$. 
\end{theorem}

Let $x,y\in \bQ^{d+1}\cap \Omega^{*}$ and $r=|x-y|$. Then $x,y\in \Omega_{j}$ for all $j$ large, so there is an $L$-bi-Lipschitz $f_{x,y,j}:\cnj{B(0,r)}\rightarrow \Omega_{j}$ such that $\{x,y\}\subseteq f_{x,y,j}(\cnj{B(0,r)})$. By passing to a subsequence, we may assume $f_{x,y,j}$ converges uniformly to an $L$-bi-Lipschitz map $f_{x,y}:\cnj{B(0,r)}\rightarrow \bR^{d+1}$. If $\ve>0$ is small enough (depending on $x$ and $y$), we may assume that for each $j$ there is $B_{j}$ of radius $\frac{\min\{r_{x},r/2\}}{8L^{2}}$ so that $B_{j}\subseteq f_{x,y,j}(B(0,(1-\ve)r))\cap B_{x}$, see Figure \ref{f:jeeze}. 
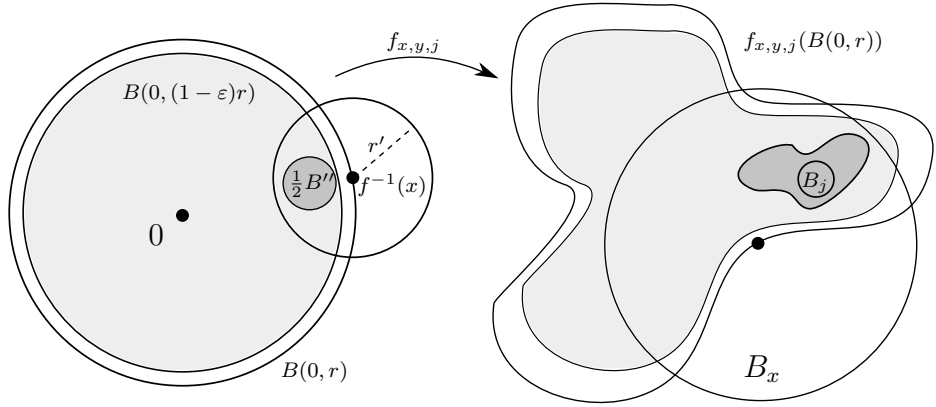
\begin{figure}[h]
\includegraphics[width=350pt]{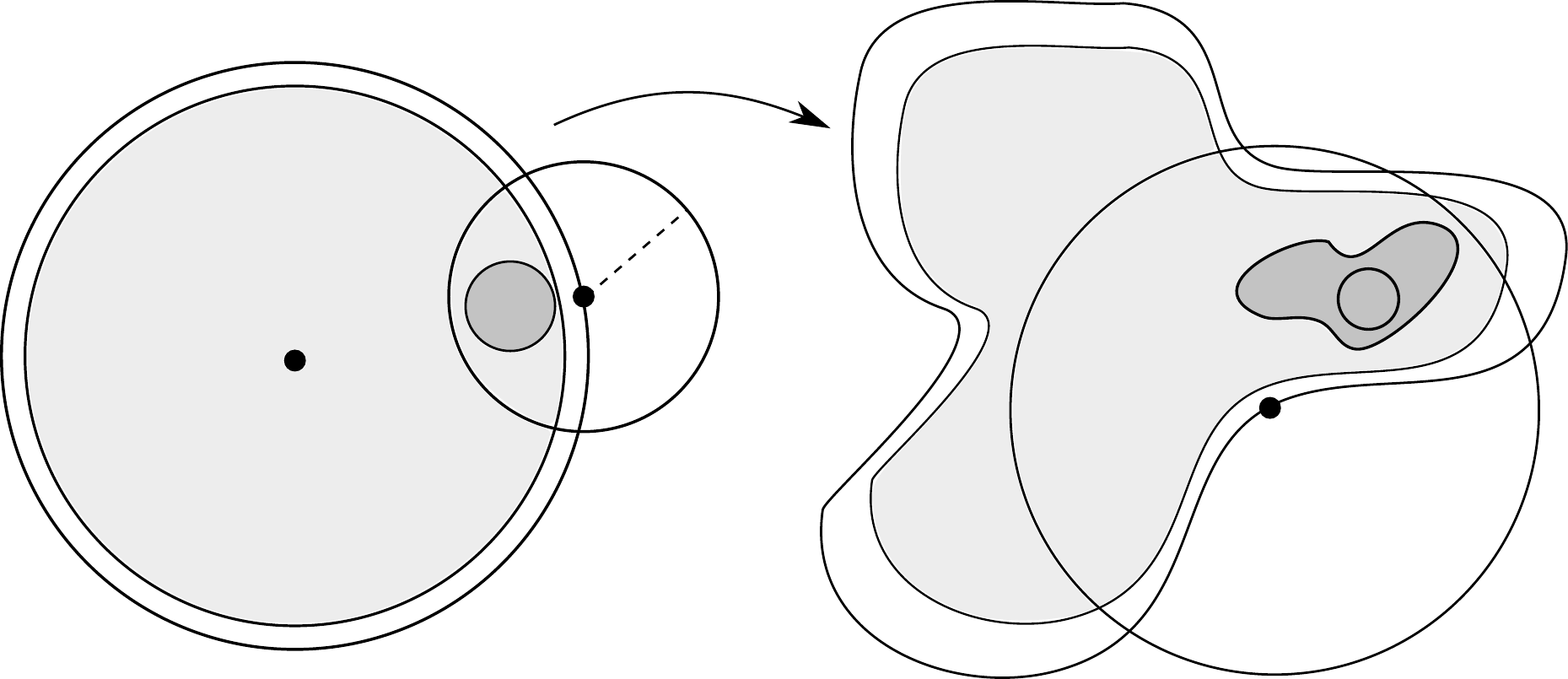}
\begin{picture}(0,0)(350,0)
\put(50,60){$0$}
\put(129,80){\tiny $f^{-1}(x)$}
\put(133,95){\tiny $r'$}
\put(103,81){\tiny $\frac{1}{2}B''$}
\put(275,10){$B_{x}$}
\put(275,135){\tiny $f_{x,y,j}(B(0,r))$}
\put(297,82){\tiny $B_{j}$}
\put(100,10){\tiny$ B(0,r)$}
\put(40,115){\tiny$ B(0,(1-\ve)r)$}
\put(140,135){\tiny $f_{x,y,j}$}
\end{picture}
\caption{Here, the outer ball on the left is $B(0,r)$, the shaded smaller ball is $B(0,(1-\ve)r)$, and the darker ball $\frac{1}{2}B''\subseteq B(0,(1-\ve)r)$ are on the left and their images are shown on the right. The image of $\frac{1}{2}B''$ is contained in $B_{x}$ and contains a smaller ball $B_{j}$ of desired radius.}
\label{f:jeeze}
\end{figure}

To see this, let $r'=\min\{r_{x},r/2\}$. Then there is $B_{j}''\subseteq B(0,r)\cap B(f_{x,y,j}^{-1}(x),r'/L)$ of radius $\frac{r'}{4L}$. Then $f(B_{j}'')\subseteq B(x,r')\subseteq B_{x}$. Now, if $\ve>0$ is small enough, $\frac{1}{2}B_{j}''\subseteq B(0,(1-\ve)r)$ and $f(\frac{1}{2}B'')$ contains a ball $B_{j}$ of radius $\frac{r'}{8L^{2}}$, which proves the claim.

By passing to a subsequence, we may assume there is a ball 
\begin{equation}
\label{e:ball-in-all}
B\subseteq  f_{x,y,j}(B(0,(1-\ve)r))\cap B_{x}\subseteq \Omega_{j}
\end{equation} 
for all $j$ large. Observe that 
\[\dist(f_{x,y,j}(\cnj{B(0,(1-\ve/2)r)}),\d\Omega_{j}) \geq \frac{\ve}{2L}r\]
 for all $j$ large, 
 \[\dist(f_{x,y}(\cnj{B(0,(1-\ve/2)r)}),\Sigma)\geq \frac{\ve}{2L}r.\] 
Hence, $f_{x,y}(\cnj{B(0,(1-\ve/2)r)}) \subseteq \Sigma^{c}$, and by uniform convergence, we have  $f_{x,y,j}(\cnj{B(0,(1-\ve)r)}) \subseteq \Sigma^{c}=\d\Omega^{*}$ for $j$ large. By \eqn{ball-in-all}, since $B\subseteq \Omega^{*}$ and since $f_{x,y,j}(\cnj{B(0,(1-\ve)r)})$ is connected, $f_{x,y,j}(\cnj{B(0,(1-\ve)r)})\subseteq \Omega^{*}$ for $j$ sufficiently large. Again, by uniform convergence, we also have $f_{x,y}(\cnj{B(0,(1-\ve)r)})\subseteq \Omega^{*}$. Letting $\ve\rightarrow 0$, we get $f_{x,y}(B(0,r))\subseteq \Omega^{*}$. Note that $x,y\in \cnj{ f_{x,y}(B(0,r))}$. 

Thus, for all $x,y\in \bQ^{d+1}\cap \Omega^{*}$, we can find an $L$-bi-Lipschitz map $f_{x,y}:\cnj{B(0,|x-y|)}\rightarrow \cnj{\Omega^{*}}$ containing $x,y$. By Arzela-Ascoli, we can find such a map for every $x,y\in \Omega^{*}$. Since balls are uniform domains and bi-Lipschitz maps preserve uniformity, we have that $f_{x,y}(B(0,|x-y|))$ is a uniform domain (with constant depending on $d$ and $L$, which in turn only depends on the uniformity constant of $\Omega$). Thus, we can find a path $\gamma$ satisfying the conditions of Definition \ref{d:uniform} for the domain $f_{x,y}(B(0,|x-y|))$, and it will also satisfy Definition \ref{d:uniform} for the domain $\Omega^{*}$. This shows that $\Omega^{*}$ is uniform.

\end{proof}

\begin{lemma}
For $\xi\in \d\Omega^{*}\backslash \supp \mu$, let $B(\xi)=B(\xi, \dist(\xi,\supp \mu)/2)$. Let $x_{\xi}$ be the center of the ball $B\subseteq B(\xi)\cap \Omega^{*}$ given by \Lemma{cork-in-limit}. Then $\omega_{j}^{x_{\xi}}(B(\xi))\rightarrow 0$, and in particular, $\delta B(\xi)\cap \Omega^{*,\ext}=\emptyset$, where $\delta$ is as in \Lemma{bourgain}.
\label{l:deltaBxi}
\end{lemma}

\begin{proof} 
Let $R\geq 1$ be so that $R\bB\supseteq B(\xi)$. Let $B'$ be the ball from \Lemma{cork-in-limit} applied to $R\bB$ and $x^{B'}$ its center. Since $\Omega_{j}$ is uniform, we may apply \Lemma{w/w} and get for $j$ large
\begin{align}
\omega_{j}^{x_{\xi}}( B(\xi))& 
\stackrel{\eqn{harnack}}{\sim_{r_{B'},R}} \omega_{j}^{x^{B'}}(B(\xi))
\leq \frac{\omega_{j}^{x^{B'}}(B(\xi))}{\omega_{j}^{x^{B'}}(R\bB)}
 \stackrel{\eqn{w/w}}{\sim} \frac{\omega_{j}^{T_{j}(x_{0})}(B(\xi))}{\omega_{j}^{T_{j}(x_{0})}(R\bB)} \notag \\
&
 \leq \frac{\omega_{j}^{T_{j}(x_{0})}(B(\xi))}{\omega_{j}^{T_{j}(x_{0})}(\bB)}
=\mu_{j}(B(\xi)).
\label{e:w<mu}
\end{align}

Thus, since $2B(\xi)\cap \supp \mu=\emptyset$, 
\[\limsup_{j\rightarrow\infty}\omega_{j}^{x_{\xi}}( B(\xi))
\stackrel{\eqn{w<mu}}{\lec}
\limsup_{j\rightarrow\infty}\mu_{j}(B(\xi))\leq \mu(2B(\xi))=0\]

Now, suppose $\delta B(\xi) \cap \Omega^{*,\ext}\neq\emptyset$. Then there is a ball $B''\subseteq \delta B(\xi) \cap \Omega_{*}^{\ext}$ with rational center so that $B''\subseteq\delta  B(\xi) \cap\Omega_{j}^{\ext}$ for all large $j$. Since we also have $B\subseteq \delta B(\xi) \cap \Omega_{j}$, this implies $\cH^{d}_{\infty}(\delta B(\xi) \cap\d\Omega_{j})\gec_{r_{B},r_{B''}} r_{\delta B(\xi)}^{d}$ for all $j$, where the implied constant depends on $r_{B}$ and $r_{B'}$ (the proof for Hausdorff measure is shown in \cite[Lemma 2.3]{Bad12}, but the same proof works for Hausdorff content) and hence by \Lemma{bourgain},
\begin{equation}
\omega_{j}^{x}( B(\xi))\gec_{r_{B},r_{B'}} 1 \mbox{ for all }x\in \delta B(\xi).
\label{e:bourgain-on-omegaj}
\end{equation}
So in particular, this holds for $x=x_{\xi}$, but that would contradict the first half of this theorem. Thus, $\delta B(\xi)\cap\Omega^{*,\ext}=\emptyset$.
\end{proof}

\begin{lemma} 
Let 
\[\tilde{\Omega}=\Omega^{*}\cup \bigcup\ck{\frac{\delta}{2} B(\xi):\xi\in \d\Omega^{*}\backslash \supp \mu}.\]
Then $\tilde{\Omega}$ is also a uniform domain and $\d\tilde{\Omega}=\supp \mu$.
\end{lemma}

\begin{proof}
Since $\Omega^{*}$ is uniform, we know that for all $x,y\in \cnj{\Omega^{*}}$ there is a good curve $\gamma$. But $\cnj{\Omega^{*}}\cap \frac{\delta}{2}  B(\xi)=\frac{\delta}{2}  B(\xi)$ for any $\xi \in \d\Omega^{*}\backslash \supp \mu$, and thus for any pair of points $x,y \in \tilde{\Omega}$ there is a good curve for $x$ and $y$ {\it with respect to} $\Omega^{*}$, and this curve will also be good for $\tilde{\Omega}$. Since there are good curves for all pairs in $\tilde{\Omega}$, it is not hard to show that there are good curves between any pair of points in its closure.

Clearly $\supp \mu\subseteq \cnj{\tilde{\Omega}}$. Moreover, $\supp\mu\subseteq \Sigma$ implies $\supp \mu\cap \Omega^{*}=\emptyset$, and by definition $\supp\mu\cap \frac{\delta}{2} B(\xi)=\emptyset$ for all $\xi\in \d\Omega^{*}\backslash \supp\mu$, hence $\supp \mu\subseteq \d\tilde{\Omega}$.  Now suppose there is $\zeta\in \d\tilde{\Omega}\backslash \supp \mu$. Since $\Omega^{*}\subseteq \tilde{\Omega}$, $\zeta\not\in \Omega^{*}$. Since $\zeta\not\in \frac{\delta}{2}  B(\xi)\cup \supp \mu$ for any $\xi\in \d\Omega^{*}\backslash\supp \mu$, we also know $\zeta\not\in \d\Omega^{*}$, and so $\zeta\in \Omega^{*\ext}$. In particular, 
\[
\zeta\in \cnj{\bigcup_{\xi\in \d\Omega^{*}\backslash\supp \mu}\frac{\delta}{2}B(\xi)}\]
Thus  there is $\xi\in \d\Omega^{*}\backslash \supp \mu$ such that $\dist(\zeta,\frac{\delta}{2}B(\xi))<\frac{\delta}{8}\dist(\zeta,\supp\mu)$, and so
\begin{align*}
|\zeta-\xi|
& \leq r_{\frac{\delta}{2}B(\xi)}+\dist(\zeta,\frac{\delta}{2}B(\xi))
 < \frac{\delta}{4}\dist(\xi,\supp\mu)+ \frac{\delta}{8}\dist(\zeta,\supp\mu) \\
& < \frac{3\delta}{8}\dist(\xi,\supp\mu)+\frac{\delta}{8}|\zeta-\xi|
\end{align*}
which implies
\[|\zeta-\xi|<(1-\delta/8)^{-1} \frac{3\delta}{8}\dist(\xi,\supp \mu) <\delta \dist(\xi,\supp \mu)/2 \]
for $\delta$ small enough. But then $\delta B\cap \Omega^{*,\ext}\neq\emptyset$, which contradicts \Lemma{deltaBxi}.

\end{proof}

\begin{lemma}
By passing to a subsequence, for any $f\in C_{0}(\bR^{d+1})$, the function
\[u_{f,j}(x)=\int fd\omega_{j}^{x}\]
converges to a harmonic function $v_{f}$ on $\Omega^{*}$ in the sense that for all compact subsets $K\subseteq \Omega^{*}$, $K\subseteq \Omega_{j}$ for $j$ sufficiently large and $u_{f,j}$ converges uniformly to $v_{f}$ on $K$. In particular, 
\begin{equation}
v_{f}=v_{g}\mbox{ for all }f,g\in C_{0}(\bR^{d+1}) \mbox{ such that }f\one_{\d\tilde{\Omega}}=g\one_{\d\tilde{\Omega}}.
\label{e:vf}
\end{equation}
\label{l:uf}
\end{lemma}

\begin{proof}
The set of continuous functions vanishing at infinity is separable in the $L^{\infty}$-metric, so let $A$ be a dense subset of $C_{0}(\bR^{d+1})$ and $f\in A$. For each $x\in \bQ^{d+1}\cap \Omega^{*}$, we can pass to a subsequence, so that $u_{f,j}$ converges uniformly on $B_{x}$ (recall \eqn{omegastar}), so by a diagonalization argument, we can guarantee $u_{f,j}$ converges uniformly on every $B_{x}$, and hence by a covering argument, on every compact subset of $\Omega^{*}$ to a harmonic function $v_{f}$. 

Note that if $B$ is a ball compactly contained in $\Sigma^{c}$, then $\omega_{j}(B)=0$ for large $j$, and so 
\[ v_{f}=\lim_{j\rightarrow\infty} \int fd\omega_{j}
=\lim_{j\rightarrow\infty} \int_{B^{c}} fd\omega_{j}.\]
Thus,
\begin{equation}
v_{f}=v_{f\one_{\d\Omega^{*}}}.
\label{e:vf1}
\end{equation}

Note that by \eqn{w<mu}, for $\xi\in \d\Omega^{*}\backslash\supp \mu$
\[\int_{B(\xi)}|f|d\omega_{j}^{x_{\xi}}\leq ||f||_{\infty} \omega_{j}^{x_{\xi}}(B(\xi))\lec \mu_{j}(B(\xi))\rightarrow 0.\]
Since $\Omega^{*}$ is uniform, and because each $x\in \Omega^{*}$ is in an open ball contained in $\Omega_{j}$ for $j$ large, we have
\[ \av{\int_{B(\xi)}fd\omega_{j}^{x}} \leq \int_{B(\xi)}|f|d\omega_{j}^{x} \sim \int_{B(\xi)}fd\omega_{j}^{x_{\xi}}\rightarrow 0.\]
Since this holds for all $\xi\in \d\Omega^{*}\backslash \supp \mu$ and \eqn{vf1} holds, we have
\begin{equation}
v_{f}=v_{f\one_{\supp \mu}}.
\end{equation}

Thus, by a diagonalization argument and the density of $A$ in $C_{c}^{\infty}(\bR^{d+1})$, we can ensure that for all $f\in C_{c}^{\infty}$, there is a harmonic function $v_{f}:\Omega^{*}\rightarrow \bR$ that is the uniform limit of $\int fd\omega_{j}$ on compact subsets of $\Omega^{*}$ and such that \eqn{vf} holds. 
\end{proof}

Combining all the previous lemmas, we have now shown the following.

\begin{lemma}
Let $\Omega$ be a uniform domain, $\xi_{0}\in \d\Omega$ and $r_{j}\rightarrow 0$ such that $\mu_{j}=T_{j\#}\omega_{\Omega}^{x_{0}}/\omega_{\Omega}^{x_{0}}(B(\xi_{0},r_{j}))$ converges weakly to a measure $\mu$. Then we may pass to a subsequence such that 

\begin{enumerate}
\item $\supp \mu$ is the boundary of a $C'$-uniform domain $\tilde{\Omega}$, where $C'$ depends on $C$ and $d$.
\item There is a uniform subdomain $\Omega^{*}$ dense in $\tilde{\Omega}$ such that for all $x\in \Omega^{*}$, if $\Omega_{j}:=T_{\xi_{0},r_{j}}(\Omega)$, then $x\in \Omega_{j}$ for all sufficiently large $j$. 
\item For $x\in \Omega^{*}$ and $\omega_{j}:=\omega_{\Omega_{j}}$, and any continuous function $f$ vanishing at infinity, $\int f d\omega_{j}$ converges to a harmonic function $v_{f}$ uniformly on compact subsets of $\Omega^{*}$ such that \eqn{vf} holds.
\end{enumerate}
\label{l:just-uniform}
\end{lemma}

\begin{lemma}
Let $\Omega\subseteq \bR^{d+1}$ be a uniform domain. For almost every non-degenerate point $\xi_{0}\in \d\Omega$, if $r_{j}\rightarrow 0$ and $\mu_{j}=T_{j\#}\omega_{\Omega}^{x_{0}}/\omega_{\Omega}^{x_{0}}(B(\xi_{0},r_{j}))$, then there is a subsequence that converges weakly to a measure $\mu$ satisfying the conclusions of the previous lemma. In addition, we have $v_{f}=\int fd\tilde{\omega}$ for $f\in C_{0}(\bR^{d+1})$ where $\tilde{\omega}$ is the harmonic measure for $\tilde{\Omega}$.
\label{l:assumend}
\end{lemma}

\begin{proof}
First note that for $\delta<\tau$, by the maximum principle,
\[
\sup_{|x-\xi|=\delta r}\omega_{B(\xi,r)\cap \Omega}^{x}(\d B(\xi,r)\cap \Omega)
\leq \sup_{|x-\xi|=\tau r}\omega_{B(\xi,r)\cap \Omega}^{x}(\d B(\xi,r)\cap \Omega).\]
Thus, $E=\bigcup E_{n}$ where
\begin{multline*}
E_{n}:=\{\xi\in\d\Omega: \mbox{ for all } r\in (0,1/n), \\
\sup_{|x-\xi|=r/n}\omega_{B(\xi,r)\cap \Omega}^{x}(\d B(\xi,r)\cap \Omega)\leq 1-1/n\} .
\end{multline*}
Fix an $n$ and let $\xi_{0}$ be a point of density for $E_{n}$ with respect to the measure $\omega_{\Omega}^{x_{0}}$.  By \Lemma{tanonempty}, we can pass to a subsequence so that $\mu_{j}$ converges weakly to a measure $\mu$, and thus again to another subsequence so that the conclusions of \Lemma{just-uniform} hold. Let $f\in C_{0}(\bR^{d+1})$, $\ve>0$, and $\xi\in \d\tilde{\Omega}$. Pick $r>0$ small enough so that 
\begin{equation}
|f(\zeta)-f(\xi)|<\ve \mbox{ whenever }|\xi-\zeta|\leq r.
\label{e:fcont}
\end{equation}
Consider the function
\[h(x)=f(\xi)+\ve + 2||f||_{\infty}\omega_{j}^{x}(\cnj{B(\xi,r)}^{c})-\int fd\omega_{j}^{x}.\]
This is harmonic on $\Omega_{j}$. We will show that $h$ is nonnegative. By \cite[Theorem 5.2.6]{AG}, it suffices to show that 
\begin{equation}
\liminf_{x\rightarrow \infty}h(x)\geq 0 \mbox{ and }\liminf_{x\rightarrow \zeta}h(x)\geq 0 \mbox{ for quasi-every } \zeta\in \d\Omega.
\label{e:quasievery}
\end{equation}

Let $B(y,R)$ be a ball containing the support of $f$. Then $||f||_{\infty}\frac{R^{d-1}}{|\cdot-y|^{d-1}}$ is a subharmonic majorant of $|f|$, and thus
\[
\av{\int f d\omega_{j}^{x}}\leq \int |f|d\omega_{j}^{x}\leq\frac{ ||f||_{\infty} R^{d-1}}{|x-y|^{d-1}}\rightarrow 0\]
as $x\rightarrow \infty$. Thus $\lim_{x\rightarrow \infty} h(x)\geq 0$, which proves the first part of \eqn{quasievery}

To prove the second part, we recall that quasi-every point $\zeta\in \d\Omega_{j}$ is regular \cite[Theorem 6.6.8]{AG}, thus we only need to show $\lim_{x\rightarrow \zeta}h(x)\geq 0$ for $\zeta\in \d\Omega_{j}$ regular.
\begin{enumerate}
\item If $\zeta\in \cnj{B(\xi,r)}$, then 
\[ \liminf_{x\rightarrow \zeta} h(x)\geq f(\xi)+\ve-\lim_{x\rightarrow \zeta} \int fd\omega_{j}^{x}
=f(\xi)-f(\zeta)+\ve \stackrel{\eqn{fcont}}{>}0.\]
\item If $\zeta\not\in \cnj{B(\xi,r)}$, then the boundary data of $\omega_{j}^{x}(\cnj{B(\xi,r)}^{c})$ is continuous at $\zeta$ and thus
\[
\liminf_{x\rightarrow \zeta} h(x)=f(\xi)+\ve +2||f||_{\infty}-f(\zeta)\geq \ve>0.\]
\end{enumerate}
Thus, we have shown that $\liminf_{x\rightarrow \zeta}h(x)\geq 0$ for $\zeta$ regular, which proves the last part of \eqn{quasievery}, and hence $h\geq 0$. 

We can similarly show that the function
\[
f(\xi)-\ve -2||f||_{\infty}\omega_{j}^{x}(\cnj{B(\xi,r)^{c}})-\int fd\omega_{j}^{x}\]
is nonpositive. Combining our estimates, we obtain that

\begin{equation}
\av{\int f d\omega_{j}^{x} -f(\xi)}\leq 2||f||_{\infty} \omega_{j}^{x}(\cnj{B(\xi,r)}^{c})+\ve \mbox{ for } x\in \Omega_{j}.
\label{e:f<w}
\end{equation}

Let $\rho\in (0,1/10)$ and $\xi\in \supp \mu$. Let $R=1+|\xi|+\rho$.

Since $\xi_{0}$ is a point of density for $E_{n}$, by \Corollary{hitsball} we can ensure that for $j$ large enough there is $\zeta_{j}\in T_{j}(E_{n})$ with 
\begin{equation}
|\xi-\zeta_{j}|<\rho r.
\label{e:xi-zetaj}
\end{equation} Setting $\xi_{j}=T_{j}^{-1}(\zeta_{j})$, we have $\xi_{j}\in E_{n}\cap B(\xi_{0},Rr_{j})$ with $|T_{j}^{-1}(\xi)-\xi_{j}|<\rho r r_{j}$. Note that by the definition of $E_{n}$ and by \Lemma{holder}, for $j$ large enough so that $\frac{1}{nr_{j}}>100 r$ we have
\[\omega_{j}^{x}(\cnj{B(\zeta_{j},(1-\rho)r)}^{c})\lec_{n} \ps{\frac{|x-\zeta_{j}|}{(1-\rho)r}}^{\alpha} \mbox{ for }x\in B(\zeta_{j},(1-\rho)r)\cap\Omega_{j}.\]

Thus, we have for $x\in \Omega_{j}\cap B(\xi,r/4)\backslash B(\xi,2\rho r)\subseteq B(\zeta_{j},(1-\rho)r)$ that 
\begin{equation}
|x-\zeta_{j}|\stackrel{\eqn{xi-zetaj}}{\leq} |x-\xi|+\rho r\leq \frac{3}{2}|x-\xi|
\end{equation} 
and
\begin{equation}
\omega_{j}^{x}(\cnj{B(\xi,r)}^{c})
\leq\omega_{j}^{x}(\cnj{B(\zeta_{j},(1-\rho)r)}^{c})
\lec \ps{\frac{|x-\zeta_{j}|}{(1-\rho)r}}^{\alpha}
\lec \ps{\frac{|x-\xi|}{r}}^{\alpha}.
\label{e:w<x/r}
\end{equation}
Combining \eqn{f<w} and \eqn{w<x/r}, we get
\[\av{\int f d\omega_{j}^{x} -f(\xi)}
\lec \ps{\frac{|x-\xi|}{r}}^{\alpha}+\ve \;\; \mbox{ if }x\in \Omega_{j}\cap B(\xi,r/4)\backslash B(\xi,2\rho r).\]
Letting $j\rightarrow \infty$, and using the fact that $x\in \Omega^{*}$ implies $x\in \Omega_{j}$ for all large $j$, we have 
\[\av{v_{f}(x)-f(\xi)}
\lec \ps{\frac{|x-\xi|}{r}}^{\alpha} +\ve\;\; \mbox{ if }x\in \Omega^{*}\cap B(\xi,r/4)\backslash B(\xi,2\rho r).\]
Now let $\rho\rightarrow 0$ and we get
\[\av{v_{f}(x)-f(\xi)}
\lec \ps{\frac{|x-\xi|}{r}}^{\alpha}+\ve\;\; \mbox{ if }x\in B(\xi,r/4)\cap \Omega^{*}.\]
Hence, 
\[ f(\xi)-\ve \leq \liminf_{x\rightarrow \xi} v_{f}(x)\leq \limsup_{x\rightarrow \xi} v_{f}(x)\leq f(\xi)+\ve .\]
Letting $\ve\rightarrow 0$, we now have 
\[\lim_{x\rightarrow \xi} v_{f}(x)=f(\xi).\]

Thus, $v_{f}$ is a harmonic function on $\Omega^{*}$ whose limits at  $\d\tilde{\Omega}\subseteq \d\Omega^{*}$ coincide with $f$. Let $\tilde{\omega}=\omega_{\tilde{\Omega}}$, $\tilde{u}_{f}=\int fd\tilde{\omega}$, and $F=\tilde{u}_{f}\one_{\d\Omega^{*}\backslash \supp \mu}+f\supp\mu$. By \eqn{vf}, $v_{f}=v_{F}$. Moreover, $v_{F}$ is harmonic in $\Omega^{*}$ and has boundary limit equal to $F$ at every regular point in $\d\Omega^{*}$. In particular, it equals $f$ everywhere on $\supp \mu=\d\tilde{\Omega}$ and equals $\tilde{u}_{f}$ at every regular point of $\d \Omega^{*}\backslash \supp \mu$ since $\tilde{u}_{f}$ is continuous on $\d\Omega^{*}\backslash\supp \mu$. Thus, $v_{f}=v_{F}=\int Fd\omega_{\Omega^{*}}$. The function $\tilde{u}_{f}$ agrees with $F$ at every boundary point of $\d\Omega^{*}$ as well, hence $\tilde{u}_{f}=\int Fd\omega_{\Omega^{*}}=v_{f}$. Therefore $f$ extends harmonically to all of $\tilde{\Omega}$ and in fact $v_{f}=\int fd\tilde{\omega}$. 
\end{proof}

\begin{lemma}
With the assumptions of \Lemma{assumend}, if $E$ is the set of $(\beta,\delta)$-non-degenerate, then for almost every point $\xi_{0}\in E$, $\Omega^{*}$ is $\Delta$-uniform with constants depending on $C,d,\delta$, and $\beta$. 
\label{l:tilde-uniform}
\end{lemma}

\begin{proof}

We will assume $\delta=\frac{1}{2}$ for simplicity. Let $B=B(\xi,r)$ be a ball with $\xi\in \d\tilde{\Omega}$, $r>0$, and let $x\in \d \frac{1}{100}B\cap \Omega^{*}$. By Lemma 4.1 in \cite{Azz14}, there is a constant $C>0$ depending only on the uniformity constant of $\Omega_{j}$ (which is the same constant for all $j$) so that for all $j$ with $\d\Omega_{j}\cap B\neq\emptyset$, there is a $C$-uniform domain $\Omega_{j}^{B}\subseteq \Omega_{j}\cap CB$ such that $B\cap \Omega_{j}\subseteq \Omega_{j}^{B}$, see Figure \ref{f:ugh}. By \Lemma{assumend}, we can pass to a subsequence and guarantee there are uniform domains $\Omega^{B,*}\subseteq \tilde{\Omega}^{B}$ (the former dense in the latter) so that $\omega_{\Omega_{j}^{B}}^{y}$ converges weakly to $\omega_{\tilde{\Omega}^{B}}^{y}$ for all $y\in \Omega^{B,*}$. By the definition of $\tilde{\Omega}^{B}$, we know $\tilde{\Omega}^{B}\subseteq CB\cap\tilde{\Omega}$ and $B\cap \tilde{\Omega}\subseteq \tilde{\Omega}^{B}$.

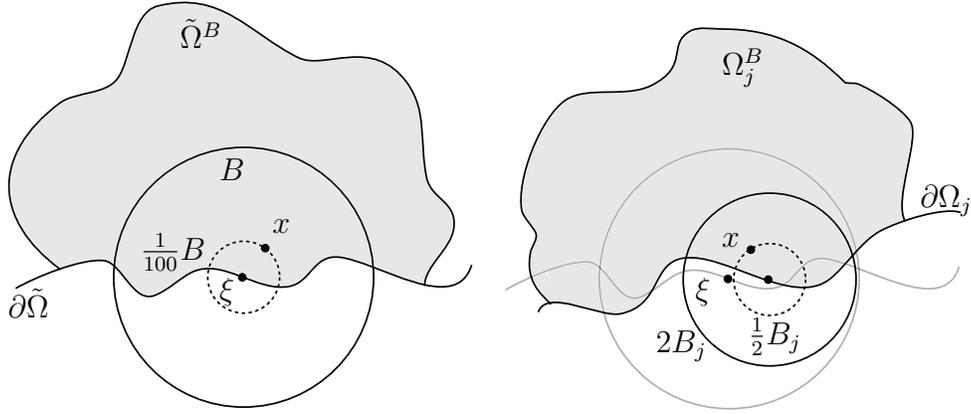
\begin{figure}[h]
\includegraphics[width=360pt]{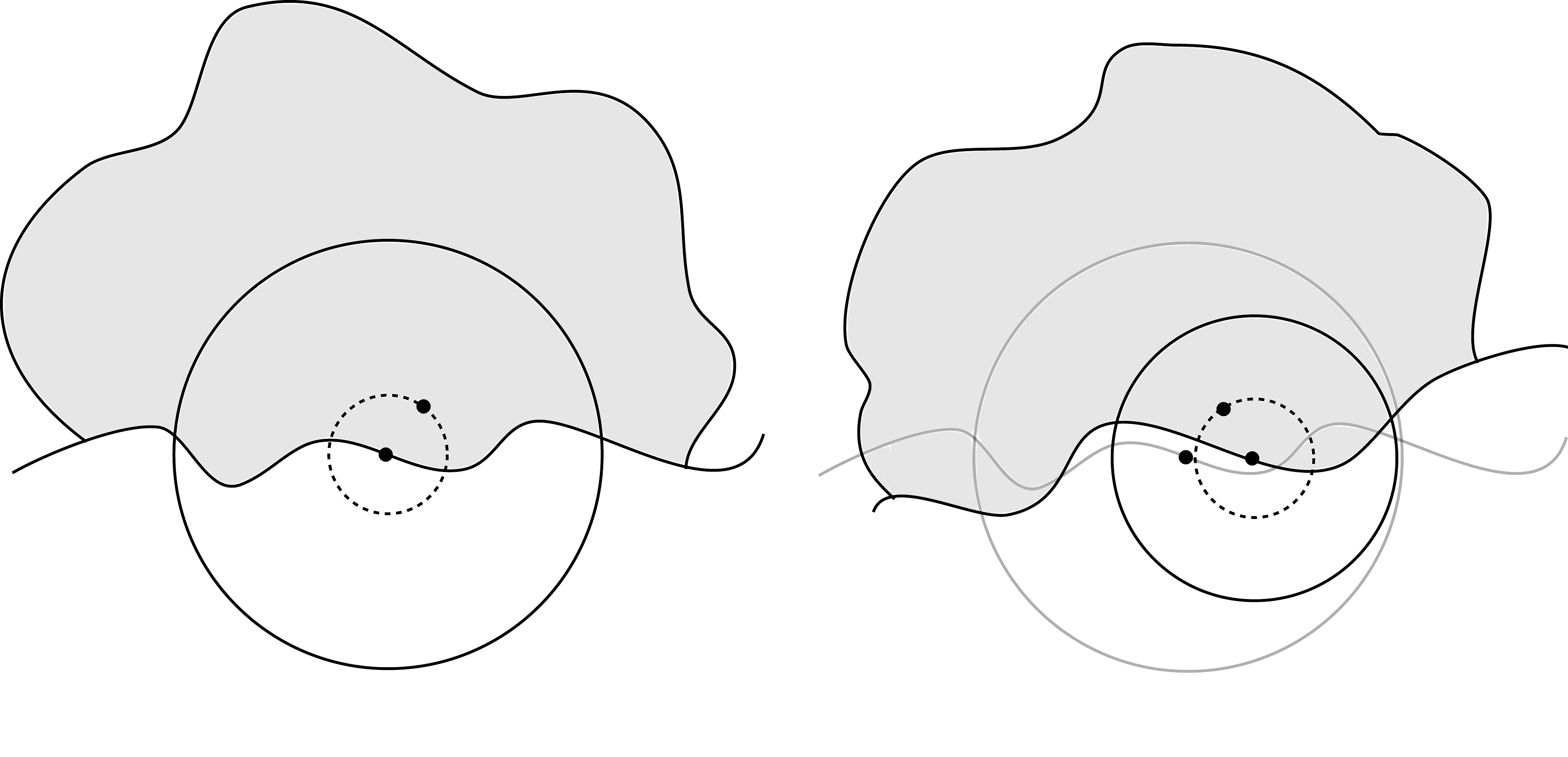}
\begin{picture}(0,0)(180,0)
\put(65,170){$\tilde{\Omega}^{B}$}
\put(80,120){$B$}
\put(50,90){$\frac{1}{100}B$}
\put(80,75){$\xi$}
\put(100,100){$x$}
\put(0,68){$\d\tilde{\Omega}$}
\put(270,160){$\Omega^{B}_{j}$}
\put(270,95){$x$}
\put(260,75){$\xi$}
\put(280,58){$\frac{1}{2}B_{j}$}
\put(245,55){$2B_{j}$}
\put(345,110){$\d\Omega_{j}$}

\end{picture}
\caption{In the figure on the left, the shaded area depicts $\Omega^{B*}\subseteq CB\cap \Omega^{*}$ and on the right we have $\Omega_{j}^{B}\subseteq CB\cap \Omega_{j}$. }
\label{f:ugh}
\end{figure}

Let $E_{n}=\{\xi\in E:\eta_{1/2}(\xi,r)<\beta \mbox{ for }r<1/n\}$. Then almost every $\xi_{0}\in E$ is a point of density for some $E_{n}$, $n\in \bN$. By \Corollary{hitsball}, for each $j$ sufficiently large we can pick $\xi_{j}\in T_{j}(E_{n})\subseteq \d\Omega_{j}$ converging to $\xi$ and set $B_{j}=B(\xi_{j},2|\xi_{j}-x|)$ so that $B_{j}\subseteq \frac{1}{2}B$, $r_{B_{j}}\geq r_{B}/1000$, and $x\in \d \frac{1}{2}B_{j}$. Then by the maximum principle, weak limits, the maximum principle again, and since $\xi_{j}\in T_{j}(E_{n})$, we have
\begin{align*}
\omega_{CB\cap \tilde{\Omega}}^{x}(\cnj{B}\cap \d\tilde{\Omega})
& \geq \omega_{\tilde{\Omega}^{B}}^{x}(\cnj{B})
\geq \limsup_{j\rightarrow\infty}\omega_{\Omega_{j}^{B}}^{x}\ps{\frac{1}{2}\cnj{B}}\\
& \geq \limsup_{j\rightarrow\infty}\omega_{\Omega_{j}\cap B_{j}}^{x}(\cnj{B_{j}}\cap \d\Omega_{j})\geq 1-\beta>0
\end{align*}
for some $\beta$ depending only on $\delta$ and $\beta$. This implies $\omega_{CB\cap \tilde{\Omega}^{B}}^{x}(\d B\cap \tilde{\Omega})\leq \beta<1$. Since $x\in \d \frac{1}{100} B$ and our choice of ball $B$ were arbitrary, we have thus shown $\Delta$-uniformity. 
\end{proof}

\begin{lemma}
With the assumptions of \Lemma{assumend}, if $B'\subseteq B=B(\xi,r)$  are balls centered on $\d\tilde{\Omega}$ and $B(x,\frac{r}{C'})\subseteq B\cap \Omega$ is a $C'$-corkscrew ball (recall $\tilde{\Omega}$ is uniform) then
\begin{equation}
\frac{\omega_{\tilde{\Omega}}^{x}(B')}{\omega_{\tilde{\Omega}}^{x}(B)}
\sim_{C,d} \frac{\mu(B')}{\mu(B)}.
\label{e:mumu}
\end{equation}
\label{l:mumu}
\end{lemma}

\begin{proof}
Let $T_{j}=T_{\xi_{0},r_{j}}$ and $\mu_{j}:= T_{j\#}\omega_{\Omega}^{x_{0}}/\omega_{\Omega}^{x_{0}}(B(\xi_{0},r_{j}))$ be the subsequence obtained in \LemmaI (note that $\mu_{j}(\bB)=1$). Let $\zeta\in \d\tilde{\Omega}$, $B=B(\zeta,R)$, $\xi\in B\cap \d\tilde{\Omega}$, and $r\in(0,R)$ so that 
\begin{equation} B'=B(\xi,r)\subseteq B(\xi,2r)\subseteq B.
\label{e:B'B}
\end{equation}
Fix $M\geq 1$ so that $2B\subseteq \frac{M}{4C} \bB$. 

Let $\xi_{j}\in \d\Omega_{j}$ converge to $0$ and $B_{j}=B(\xi_{j},M-|\xi_{j}|)$, so for $j$ large we have $\frac{M}{2}\bB\subseteq B_{j}\subseteq M\bB$. Let $y_{j}$ be a corkscrew point for $B_{j}$ in $\Omega_{j}$, so $B(y_{j},r_{B_{j}}/C)\subseteq B_{j}\cap \Omega_{j}$. By passing to a subsequence if necessary, and since $r_{B_{j}}\rightarrow M$, we can assume there is $y$ so that $B(y,\frac{M}{2C})\subseteq \Omega_{j}\cap B_{j}$ for all $j$ large enough. Since $2B\subseteq \frac{M}{4C}\bB$, we know $y\in \Omega_{j}\backslash 2B$, and for $j$ large enough we know $x_{j}:=T_{j}(x_{0})\in \Omega_{j}\backslash M\bB\subseteq \Omega_{j}\backslash MB_{j}$, so we can apply \Lemma{w/w} twice to get for $\ve\in (0,1)$

\begin{align*}
\frac{\omega_{\tilde{\Omega}}^{x}(B')}{\omega_{\tilde{\Omega}}^{x}(B)}
& \stackrel{\eqn{w/w}}{\sim} \frac{\omega_{\tilde{\Omega}}^{y}(B')}{\omega_{\tilde{\Omega}}^{y}(B)}
\leq \liminf_{j\rightarrow\infty}\frac{\omega_{\Omega_{j}}^{y}(B')}{\omega_{\Omega_{j}}^{y}((1-\ve)B)}\\
&  \stackrel{\eqn{w/w}}{\sim}  \liminf_{j\rightarrow\infty}\frac{\omega_{\Omega_{j}}^{x_{j}}(B')/\omega_{\Omega_{j}}^{x_{j}}(B_{j})}{\omega_{\Omega_{j}}^{x_{j}}((1-\ve)B))/\omega_{\Omega_{j}}^{x_{j}}(B_{j})}\\
& =\liminf_{j\rightarrow\infty}\frac{T_{j\#}\omega_{\Omega}^{x_{0}}(B')}{T_{j\#}\omega_{\Omega}^{x_{0}}((1-\ve)B))} \\ 
& =\liminf_{j\rightarrow\infty}\frac{\mu_{j}(B')}{\mu_{j}((1-\ve)B))}
\leq \frac{\mu(\cnj{B'})}{\mu((1-\ve)B)}
\end{align*}
Letting $\ve\rightarrow 0$, we get
\[
\frac{\omega_{\tilde{\Omega}}^{x}(B')}{\omega_{\tilde{\Omega}}^{x}(B)}
\leq \frac{\mu(\cnj{B'})}{\mu(B)}.\]
Now apply this to $\rho B'$ and take $\rho\uparrow 1$, we get 
\[
\frac{\omega_{\tilde{\Omega}}^{x}(B')}{\omega_{\tilde{\Omega}}^{x}(B)}
=\lim_{\rho\uparrow 1}\frac{\omega_{\tilde{\Omega}}^{x}(\rho B')}{\omega_{\tilde{\Omega}}^{x}(B)}
\lec \lim_{\rho\rightarrow 1}\frac{\mu(\rho \cnj{B'})}{\mu(B)}
=\frac{\mu({B'})}{\mu(B)}.\]
Thus, we get  one inequality in \eqn{mumu}. The other inequality has a similar proof. 


\end{proof}

\begin{lemma}
With the assumptions of \Lemma{assumend}, suppose there is $E\subseteq \d\Omega$ with $\omega_{\Omega}^{x_{0}}(E)>0$ and $c>0$ so that
\begin{equation}
\liminf_{r\rightarrow 0} \frac{\cH^{s}_{\infty}(B(\xi,r)\cap \d\Omega)}{r^{s}}\geq c \mbox{ for all }\xi\in E.
\end{equation}
Then for $\omega_{\Omega}^{x_{0}}$-almost every $\xi_{0}\in E$, there is $c'>0$ depending on $s,d$ and $c$ so that 
\begin{equation}
\cH^{s}_{\infty}(B(\xi,r)\cap \d\tilde{\Omega})\geq c'r^{s} \mbox{ for all }\xi\in \d\tilde{\Omega }\mbox{ and }r>0.
\end{equation}
\label{l:contentlemma}
\end{lemma}

\begin{proof}
Let $\xi\in \d\tilde{\Omega}$ and $r>0$. Set 
\[E_{n}=\{\xi\in \d\Omega: \cH^{s}_{\infty}(B(\xi,r)\cap \d\Omega)>r^{s}\mbox{ for }r<n^{-1}\}.\]
Then $E=\bigcup E_{n}$. Let $\xi_{0}$ be a point of density in some $E_{n}$ with respect to $\omega_{\Omega}^{x_{0}}$. Let $\xi\in \d\tilde{\Omega}=\supp \mu$, $r>0$. Then by \Corollary{hitsball} there is $\xi_{j}\in E_{n} \cap T_{j}^{-1}(B(\xi,r/2))$, and thus if $j$ is large enough so that $rr_{j}/2<1/n$,
\begin{align*}
\cH^{s}_{\infty}(B(\xi,r)\cap \d\Omega_{j})
& = r_{j}^{-s}\cH_{\infty}^{s}(T_{j}^{-1}(B(\xi,r))\cap \d\Omega)\\
& \geq r_{j}^{-s}\cH^{d}_{\infty}(B(\xi_{j},rr_{j}/2)\cap \d\Omega)
\geq \frac{cr^{s}}{2^{s}}
.\end{align*}
Let $\nu_{j}$ be an $s$-Frostmann measure with support in $B(\xi,r)\cap \d\Omega_{j}$ so that $\nu_{j}(B(\xi,r))\gec \frac{cr^{s}}{2^{s}}$. By passing to a subsequence, we can assume $\nu_{j}$ converges weakly to another $s$-Frostmann measure $\nu$ and $\nu(\cnj{B(\xi,r)})\gec \frac{cr^{s}}{2^{s}}$. If $\zeta\in \supp \nu$, then for all $t>0$, $\nu_{j}(B(\zeta,2t))\geq \nu(B(\zeta,t))/2>0$ for $j$ large enough, and so for $j$ large
\[
\cH^{s}_{\infty}(B(\zeta,2t)\cap \d\Omega_{j})
\gec \nu_{j}(B(\zeta,2t))\geq \nu(B(\zeta,t))/2>0.\]
Thus, there is $\zeta_{j}\in \d\Omega_{j}\cap B(\zeta,2t)$, and so
\[ \cH^{s}_{\infty} (B(\zeta_{j},4t)\cap \d\Omega_{j})/(4t)^{s}\geq \frac{\nu(B(\zeta,t))}{2(4t)^{s}}>0.\]
Hence, by \Lemma{bourgain}, for all $j$ large,

\begin{align*}
\omega_{j}^{x_{B(\zeta_{j},4t)}} & (B((\zeta,4t(1+\delta^{-1})))
\geq \omega_{j}^{x_{B(\zeta_{j},4t)}}(B(\zeta_{j},4t\delta^{-1}))\\
& \gec  \cH^{s}_{\infty} (B(\zeta_{j},4t)\cap \d\Omega_{j})/(4t)^{s}
  \gec \frac{\nu(B(\zeta,t))}{2(4t)^{s}}>0.\end{align*}
and hence $\tilde{\omega}^{x_{B(\zeta_{j},4t)}}(B((\zeta,4t(1+\delta^{-1})))>0$ for all $t>0$, which implies $\zeta\in \supp \tilde{\omega}=\d\tilde{\Omega}$. This implies, finally, that
\[
\cH^{s}_{\infty}(B(\xi,r)\cap \d\tilde{\Omega})
\gec \nu(B(\xi,r))\gec \frac{cr^{s}}{2^{s}}.\]
Since this holds for all $\xi\in \d\tilde{\Omega}$ and $r>0$, this finishes the proof.

\end{proof}

This finishes the proof of \LemmaI

\section{Proof of \TheoremI}

In this section, all implied constants are assumed to depend on the uniformity constant and $d$. Let 
\[F=\{\xi\in \d\Omega:0
<\theta^{\alpha,*}(\omega_{\Omega}^{x_{0}},\xi)<\infty, \;\;\; \xi \mbox{ non-degenerate}\}.\]

We fix $\xi_{0}\in F$ such that the conclusions of \Lemma{upperlim} and  \LemmaI hold for $\mu_{j}=T_{\xi_{0},r_{j}\#}\omega^{x_{0}}_{\Omega}/\omega_{\Omega}^{x_{0}}(B(\xi_{0},r_{j}))$. Then for balls $B'\subseteq B$ centered on $\d\tilde{\Omega}$,
\begin{equation}
\frac{\omega_{\tilde{\Omega}}^{x_{B}}(B')}{\omega_{\tilde{\Omega}}^{x_{B}}(B)}\sim \frac{\mu(B')}{\mu(B)}\lesssim \frac{r_{B'}^{\alpha}}{\mu(B)}.\label{e:wsimr^a}
\end{equation}


Pick $\tilde{B}\subseteq \tilde{\Omega}$ so that there is $\zeta\in \d\tilde{B}\cap \d\tilde{\Omega}$. Let $\tilde{x}$ be the center of $\tilde{B}$. We claim that if $\alpha>d$, then the normal derivative of $G_{\tilde{B}}(\tilde{x},\cdot)$ at $\zeta$ is zero. Let $x\in [\zeta,\tilde{x}]\cap \d B(\zeta,r_{\tilde{B}}/2)$. Let $B=B(\zeta,2r_{\tilde{B}})$ and $B'=B(\zeta,|x-\zeta|)$, see Figure \ref{f:fml}.

\begin{figure}[h]
\includegraphics[width=200pt]{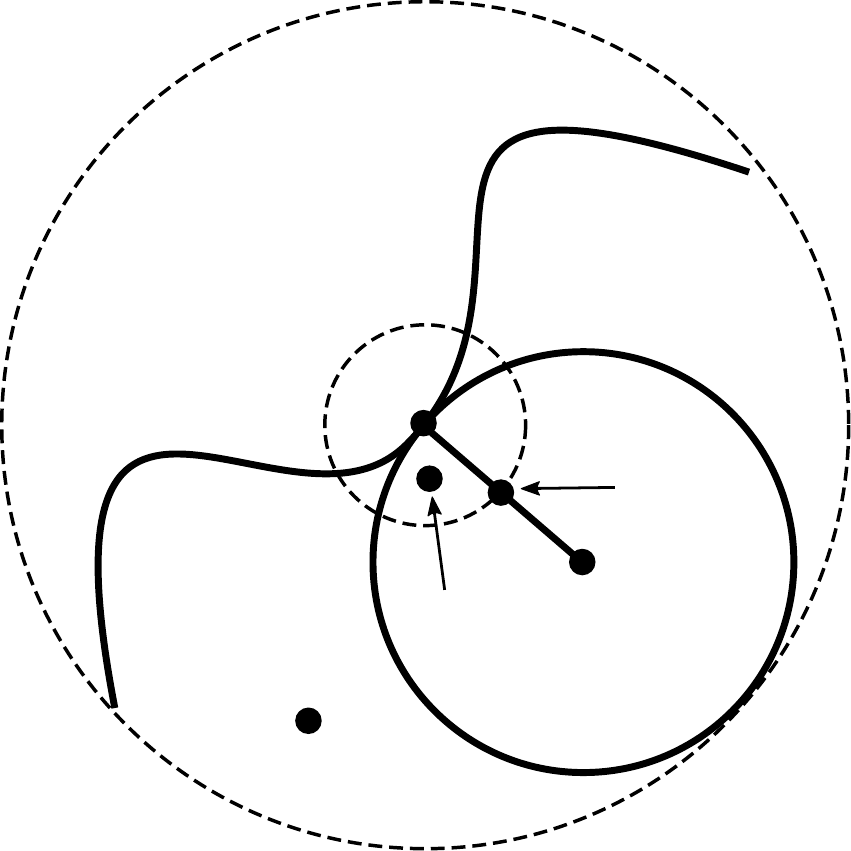}
\begin{picture}(0,0)(200,0)
\put(142,83){$x$}
\put(95,55){$x_{B'}$}
\put(130,55){$\tilde{x}$}
\put(175,100){$\tilde{B}$}
\put(25,150){$B$}
\put(75,125){$B'$}
\put(65,20){$x_{B}$}
\put(3,80){$\d\tilde{\Omega}$}
\put(90,110){$\zeta$}
\end{picture}
\caption{The balls $\tilde{B}$, $B$, and $B'$ }
\label{f:fml}
\end{figure}
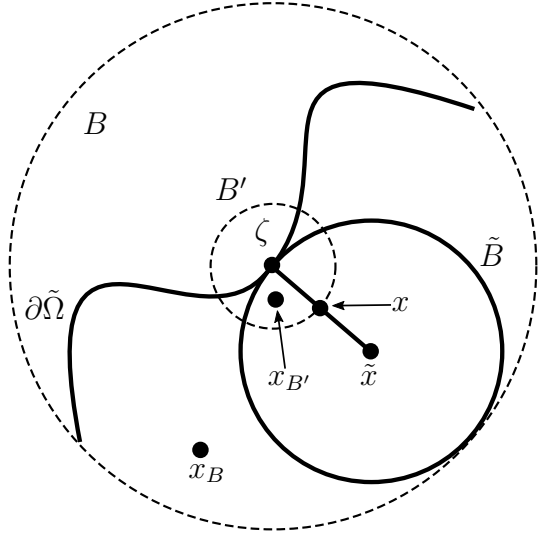

Since $G_{\tilde{B}}(\tilde{x},\zeta)=0$ and because $G_{\tilde{B}}\leq G_{\tilde{\Omega}}$ by the maximum principle, we get
\[
\frac{|G_{\tilde{B}}(\tilde{x},x)-G_{\tilde{B}}(\tilde{x},\zeta)|}{|x-\zeta|}
 =\frac{G_{\tilde{B}}(\tilde{x},x)}{|x-\zeta|}
\leq \frac{G_{\tilde{\Omega}}(\tilde{x},x)}{|x-\zeta|}.
\]
Now we apply the Harnack chain condition in each variable of the Green function and use \Lemma{w<G} to get
\[\frac{G_{\tilde{\Omega}}(\tilde{x},x)}{|x-\zeta|}
\sim \frac{G_{\tilde{\Omega}}(x_{B},x_{B'})}{|x-\zeta|}\\
\sim \frac{|x-\zeta|^{1-d}\omega_{\tilde{\Omega}}^{x_{B}}(B')}{|x-\zeta|\omega_{\tilde{\Omega}}^{x_{B}}(B)}= \frac{\omega_{\tilde{\Omega}}^{x_{B}}(B')}{|x-\zeta|^{d}\omega_{\tilde{\Omega}}^{x_{B}}(B)}.
\]
Finally, by \eqn{wsimr^a}, we get
\[ \frac{\omega_{\tilde{\Omega}}^{x_{B}}(B')}{|x-\zeta|^{d}\omega_{\tilde{\Omega}}^{x_{B}}(B)}
\lec \frac{r_{B'}^{\alpha}}{|x-\zeta|^{d}\mu(B)}
=\frac{|x-\zeta|^{\alpha-d}}{\mu(B)}.\]
Combining these estimates, we get
\[
\frac{|G_{\tilde{B}}(\tilde{x},x)-G_{\tilde{B}}(\tilde{x},\zeta)|}{|x-\zeta|}
\lec  \frac{|x-\zeta|^{\alpha-d}}{\mu(B)}\]
 so  as $x\rightarrow \zeta$ along $[\zeta,\tilde{x}]$, this shows that the normal derivative at $\zeta$ must be zero, as wished. But $G_{\tilde{B}}(\tilde{x},\cdot)=|\tilde{x}-\cdot|^{1-d}-r_{\tilde{B}}^{1-d}$ on $\tilde{B}$, which clearly has nonzero normal derivative at $\zeta$, and this gives a contradiction. Thus, $\alpha\leq d$. \\

\section{Proof of \TheoremII}
First assume $\theta^{\alpha}_{*}(\omega_{\Omega}^{x_{0}},\xi)\in (0,\infty) <\infty$ for each $\xi\in E$ and $\omega_{\Omega}^{x_{0}}(E)>0$. Then it is not hard to show that $E$ has $\sigma$-finite $\cH^{\alpha}$-measure. Indeed, note that if 
\[E_{k,\ell}=\{\xi\in E: \omega_{\Omega}^{x_{0}}(B(\xi,r))>r^{\alpha}/\ell \mbox{ for }r\in (0,k^{-1}]\}\]
then $E=\bigcup_{k,\ell}E_{k}$. Fix $k\in \bN$ and let $r<k^{-1}$ By the Besicovitch covering theorem, we may find a covering of $E_{k,\ell}$ by balls $B_{j}$ of bounded overlap of radii $r$ so that each $B_{j}$ is centered on $E_{k,\ell}$. Then
\[
\cH^{\alpha}_{r}(E_{k,\ell})\leq \sum r_{B_{j}}^{\alpha}\leq k\sum \omega_{\Omega}^{x_{0}}(B_{j})\lec_{d} k\omega_{\Omega}^{x_{0}}\ps{\bigcup B_{j}}\leq 1.\]
Letting $r\rightarrow 0$ shows $E_{k,\ell}$ has finite $\alpha$ measure. If $\alpha\leq d-1$, then each $E_{k,\ell}$ has finite $(d-1)$-measure. This implies $E_{k,\ell}$ is polar \cite[Theorem 5.9.4]{AG} and polar sets have harmonic measure zero \cite[Theorem 6.5.5]{AG}, thus $\omega(E_{k,\ell} )=0$ for each $k,\ell$, and hence $\omega_{\Omega}^{x_{0}}=0$, we get a contradiction since $\omega_{{\Omega}}^{x_{0}}(E)>0$. Hence $\alpha>d-1$.

Now assume \eqn{content}. Note that \eqn{content} and \Lemma{bourgain} imply each $\xi\in E$ is non-degenerate. Again, by \LemmaI, we can find a tangent measure and domain $\tilde{\Omega}$ satisfying 
\begin{equation}
\frac{\omega_{\tilde{\Omega}}^{x_{B}}(B')}{\omega_{\tilde{\Omega}}^{x_{B}}(B)}\sim \frac{\mu(B')}{\mu(B)}\gec \frac{r_{B'}^{\alpha}}{r_{B}^{\alpha}}.\label{e:wsimr^a2}
\end{equation}
and so that condition (5) of \LemmaI holds. This implies $\dim \d\tilde{\Omega}\leq \alpha$, but condition (5) implies $\dim\d\tilde{\Omega}\geq s$, and so $\alpha\geq s$.

%

\def\cprime{$'$}

\end{document}